\documentclass[12pt,reqno]{amsart}
\usepackage{amssymb,amsthm,amsmath,amsxtra,dsfont,appendix,bbm,stix,xparse}
\usepackage{hyperref} 
\usepackage{mathrsfs}

\numberwithin{equation}{section}
\makeatletter

\makeatletter
\def\@tocline#1#2#3#4#5#6#7{\relax
  \ifnum #1>\c@tocdepth 
  \else
    \par \addpenalty\@secpenalty\addvspace{#2}%
    \begingroup \hyphenpenalty\@M
    \@ifempty{#4}{%
      \@tempdima\csname r@tocindent\number#1\endcsname\relax
    }{%
      \@tempdima#4\relax
    }%
    \parindent\z@ \leftskip#3\relax \advance\leftskip\@tempdima\relax
    \rightskip\@pnumwidth plus4em \parfillskip-\@pnumwidth
    #5\leavevmode\hskip-\@tempdima
      \ifcase #1
       \or\or \hskip 1em \or \hskip 2em \else \hskip 3em \fi%
      #6\nobreak\relax
      \dotfill
      \hbox to\@pnumwidth{\@tocpagenum{#7}}
    \par
    \nobreak
    \endgroup
  \fi}
\makeatother

\theoremstyle{plain}
\newtheorem{theorem}{Theorem}[section]
\newtheorem{cor}[theorem]{Corollary}
\newtheorem{lemma}[theorem]{Lemma}
\newtheorem{remark}[theorem]{Remark}
\newtheorem{prop}[theorem]{Proposition}

\theoremstyle{definition}

\usepackage{colonequals}

\usepackage{color}

\newcommand{\LR}[1]{\left(#1\right)}

\renewcommand{\d}{\differential}

\newcommand{\Done}{\mathcal{D}^{\rm 1D}}
\newcommand{\im}{\mathrm{i}}
\newcommand{\bx}{\mathbf{x}}

\newcommand{\R}{\mathbb{R}}

\newcommand{\lone}{\lambda^{\rm 1D}}
\newcommand{\psione}{\psi^{\rm 1D}}
\usepackage{physics}

\title{Dimensional reduction for a system of 2D anyons}

\author[N. Rougerie]{Nicolas Rougerie}
\address{Ecole Normale Sup\'erieure de Lyon \& CNRS,  UMPA (UMR 5669)}
\email{nicolas.rougerie@ens-lyon.fr}

\author[Q. Yang]{Qiyun Yang}
\address{Ecole Normale Sup\'erieure de Lyon, UMPA (UMR 5669)}
\email{qiyun.yang@ens-lyon.fr}

\date{December, 2023}

\begin{document}

\maketitle

\begin{abstract}    
Anyons with a statistical phase parameter $\alpha\in(0,2)$ are a kind of quasi-particles that, for topological reasons, only exist in a 1D or 2D world. 
We consider the dimensional reduction for a 2D system of anyons in a tight wave-guide. More specifically, we study the 2D magnetic-gauge picture model with an imposed anisotropic harmonic potential that traps particles much stronger in the $y$-direction than in the $x$-direction. 
We prove that both the eigenenergies and the eigenfunctions are asymptotically decoupled into the loose confining direction and the tight confining direction during this reduction. 
The limit 1D system for the $x$-direction is given by the impenetrable Tonks-Girardeau Bose gas, which has no dependency on $\alpha$, and no trace left of the long-range interactions of the 2D model.
\end{abstract}


\tableofcontents


\newpage

\section{Introduction}

\subsection{Motivation}

Fundamental particles are sorted into two symmetry classes: bosons (e.g. photons and other force carriers) and fermions (e.g. electrons, quarks and other matter particles).
Bosons follow Bose-Einstein statistics, while fermions obey Fermi-Dirac statistics and hence comply with the Pauli exclusion principle. 
For topological reasons, particles with statistics in between those of bosons and fermions cannot exist except in a one or two-dimensional system. Such so-called anyons can thus exist only as emergent quasi-particles.  
Recent experiments~\cite{BarEtalFev-20,NakEtalMan-20} provided firm evidence that the charge carriers in fractional quantum Hall systems~\cite{Goerbig-09,Laughlin-99,Jain-07,Rougerie-hdr} behave like anyons. 
This is due to the strong correlation between electrons under a strong magnetic field, leading to the formation of incompressible quantum liquids such as the Laughlin state. Charge carriers couple to excitations thereof (Laughlin quasi-particles or quasi-holes), and thereby can acquire exotic quantum statistics.
In these experiments, anyons are observed via transport along one-dimensional edges of the two-dimensional samples, which provides part of our motivation to study the dimensional reduction for 2D anyons trapped in 1D.

There is essentially a single agreed-upon model for 2D anyons, which was derived formally as an emergent description in fractional quantum Hall physics~\cite{AroSchWil-84,LunRou-16,LamLunRou-22}. On the other hand, there are several different inequivalent models for 1D anyonic behavior~\cite{Myrheim-99,Ouvry-07,Girardeau-06,AndBarJon-99,FreMoo-21,Moosavi-22,Rabello-96,AglGriJacPiSem-96,Kundu-99,BatGuaOel-06,GreSan-15,TanEggPel-15,Bon-etal-Pos-21}. Here we prove that, to leading order, 2D anyons confined to a 1D wave-guide always behave as fermions, independently of the original statistical parameter $\alpha \in (0,2)$. The different proposed 1D anyons models hence seem to describe each a physics of their own, a priori unrelated to 2D anyons, but accessible to different emergence phenomena~\cite{EdmEtalOhb-13,ValWesOhb-20,CarGreSan-16,ChiEtalCel-22,FroEtalTar-22}. Our results were announced in~\cite{RouYan-23a}, where the reader can find more physics background and discussion.

\subsection{The model}

Formally, one might want to describe $N\geq 2$ anyons by a (multi-valued) complex many-body wave-function $\Psi\in L^2(\mathbb{R}^{dN})$ in a $d$-dimensional space $\mathbb{R}^d$ satisfying
\begin{equation} \label{anyf}
    \Psi (\dots, \mathbf{x}_j,\dots, \mathbf{x}_k,\dots) = e^{\mathrm{i} \pi \alpha} \Psi (\dots, \mathbf{x}_k,\dots, \mathbf{x}_j,\dots)
\end{equation}
for some constant $\alpha$. The cases $\alpha = 0,1$ correspond to usual bosons and fermions respectively, and one can restrict to $\alpha \in [0,1)$ by periodicity and complex conjugation. Free, non-relativistic, 2D anyons of statistics parameters $\alpha \neq 0,1$ would be described by acting on such wave-functions with a usual Schr\"odinger Hamiltonian
$$ 
\sum_{j=1}^N \left(- \Delta_{\mathbf{x}_j} + V (\mathbf{x}_j) \right).
$$
Studying the spectral properties of this operator on an appropriate domain incorporating a symmetry such as~\eqref{anyf} would require modelling wave-functions as sections of line bundles, a formalism that one often prefers to bypass for analysis purposes~\cite{Myrheim-99,LunSol-13a}. This means using the so-called ``magnetic gauge picture'' as we explain next.

Let us consider the 2D system of anyons represented by the Hamiltonian 
\begin{equation*}
    \mathcal{H}^{2\mathrm{D}}_{\varepsilon} \colonequals \sum_{j=1}^N  \left( -\Delta_{\mathbf{x}_j} + V_{\varepsilon}(\mathbf{x}_j) \right) 
\end{equation*}
defined on the domain of anyonic wave ``functions''~\cite{Myrheim-99}, where, for some $\varepsilon>0$
\begin{equation*}\label{eq:pot}
 V_{\varepsilon}{(\mathbf{x})} =  V_{\varepsilon}{(x,y)} = {x}^2 + \frac{1}{\varepsilon^2} {y}^2. 
\end{equation*}
We write an anyonic wave ``function'' $\Phi$ as
\begin{equation*}\label{rewrite}
    \Phi (\mathbf{x}_1,\dots,\mathbf{x}_N) = \LR{\prod_{1\le j < k \le N} e^{\mathrm{i}\alpha\arg(\mathbf{x}_j-\mathbf{x}_k)} } \Psi(\mathbf{x}_1,\dots,\mathbf{x}_N),
\end{equation*}
where $\Psi$ is some symmetric (i.e. bosonic) function in $ L^2_{\mathrm{sym}}(\mathbb{R}^{2N})$ and $\arg\mathbf{x}$ is the angle of $\mathbf{x}$ as a two-dimensional vector in the plane. 
The multi-valued phase related to this angle can be written as 
\begin{equation*}
    e^{\mathrm{i}\alpha\arg\mathbf{x}} = \frac{z^{\alpha}}{\abs{z}^{\alpha}} \quad \text{for} \quad \mathbf{x} = (x,y) \quad \text{and}  \quad z = x +\mathrm{i}y, 
\end{equation*}
and its derivative with respect to $\mathbf{x}$ on a branch of $z^{\alpha}$ is
\begin{equation*}
    \grad e^{\mathrm{i}\alpha\arg\mathbf{x}}  = \mathrm{i} \alpha  \frac{\mathbf{x}^{\perp}}{\abs{\mathbf{x}}^2} e^{\mathrm{i}\alpha\arg\mathbf{x}}
\end{equation*}
with 
$$\mathbf{x}^{\perp} = (-y,x).$$
With this observation, one can check that acting on $\Phi$ with $\mathcal{H}^{2\mathrm{D}}_{\varepsilon}$ (anyonic gauge-picture) is formally equivalent to acting on the corresponding $\Psi$ with the modified Hamiltonian (magnetic gauge-picture) 
\begin{equation}\label{defH2D}
H^{2\mathrm{D}}_{\varepsilon} \colonequals \sum_{j=1}^N \left({ D_j^2 + V_{\varepsilon}(\mathbf{x}_j)}\right) =  \sum_{j=1}^N \left({ \left( -\im \grad_{\bx_j} + \alpha \mathbf{A}_j (\mathbf{x}_j)\right)^2 + V_{\varepsilon}(\mathbf{x}_j)}\right),
\end{equation}
where 
\begin{equation*}\label{eq:vecpot}
 \mathbf{A}_j (\mathbf{x}) \colonequals \sum_{k\neq j} \frac{\LR{\mathbf{x}-\mathbf{x}_k}^{\perp}}{\absolutevalue{\mathbf{x}-\mathbf{x}_k}^2},
\end{equation*}
i.e. it is equivalent to study the operator $H^{\mathrm{2D}}_\varepsilon$ defined on $L^2_{\mathrm{sym}}(\mathbb{R}^{2N})$. Note that $\mathbf{A}_j$ is a function on $\mathbb{R}^{2N}$ instead of on $\mathbb{R}^2$, we write it in this form for the sake of simplicity.
The effective gauge vector potential $\mathbf{A}_j$ is associated with a magnetic field perpendicular to the plane, of magnitude 
\begin{equation*}
\operatorname{curl} \mathbf{A}_j  \colonequals   \grad^{\perp}_{\mathbf{x}} \cdot {\mathbf{A}_j} = \sum\limits_{k\neq j} \Delta_{\mathbf{x}} \ln{\absolutevalue{\mathbf{x}-\mathbf{x}_k}} = 2\pi \sum\limits_{k\neq j}\delta(\mathbf{x}-\mathbf{x}_k),
\end{equation*}  
where we denoted 
$$ 
\grad^{\perp}_{\mathbf{x}} = (-\partial_{y}, \partial_{x} ).
$$
Hence, in this model, particles feel one another as carrying Aharonov-Bohm magnetic fluxes of intensity $2\pi\alpha$. This generates the long-range magnetic interactions in~\eqref{defH2D}. Throughout the paper, we use the above magnetic gauge picture for our system of 2D anyons, as in e.g.~\cite{LunSol-13a,LunSol-13b,LunSol-14}.

We investigate the most stable states of the above Hamiltonian, i.e. the low-lying energy spectrum. Letting the parameter $\varepsilon$ in $H^{2\mathrm{D}}_{\varepsilon}$ go to $0$, intuitively, the trapping potential $V_{\varepsilon}$ will force particles together around the one-dimensional line $y=0$, and the vector potential $\mathbf{A}_j$ will keep particles apart. We find that the Tonks-Girardeau gas is the correct limit 1D system on $y=0$, i.e. we find a model of impenetrable bosons in one dimension, described by the free Hamiltonian
\begin{equation*}
    H^{1\mathrm{D}} \colonequals \sum_{j=1}^N \LR{-\partial_{x_j}^2 + x_j^2} 
\end{equation*}
but restricted to the domain of symmetric functions vanishing on diagonals, 
\begin{equation*}\label{eq:dom1D}
\Done := \left\{ \psi \in H^2 (\R^N), \, \psi \equiv 0 \mbox{ on } \mathbb{D}^{1\mathrm{D}} \right\},
\end{equation*}
where 
$$ \mathbb{D}^{1\mathrm{D}} \colonequals \left\{ (x_1,\dots,x_N) \in \mathbb{R}^{N} : \; x_j = x_l  \text{ for some } j\neq l \right\}.$$
One can find more details in Section~\ref{sec:prelim} below. 

We note that the repulsive effect of the vector potential $\mathbf{A}_j$ is well-documented in the literature (see e.g.~\cite{LunSol-13a,LunSol-13b,LunSol-14,LarLun-16,GirRou-22} for rigorous results). In our situation of strong confinement along one spatial direction, this repulsion is so strong that it forces a kind of hard-core exclusion upon the particles. This is in some sense our main finding here.

The movement in the $y$ direction will be, in the $\varepsilon\to 0$ limit, frozen in the ground state of the harmonic oscillator
\begin{equation*}\label{defHO}
    H^{\mathrm{HO}}_{\varepsilon} \colonequals - \partial_y^2 +\frac{y^2}{\varepsilon^2},
\end{equation*}
namely 
\begin{equation} \label{defu}
u_{\varepsilon}(y) \colonequals \LR{{\sqrt{\pi\varepsilon}} }^{-\frac{1}{2}} e^{-\frac{y^2}{2\varepsilon}}.
\end{equation}
The corresponding eigenvalue (ground state energy) is
$$e_{\varepsilon} \colonequals  \frac{1}{\varepsilon}. $$

\subsection{Main results and discussion}

We define our Hamiltonians as Friedrichs extensions of non-negative quadratic forms. Regarding the 2D Hamiltonian, we close the quadratic form starting from the core
$$ C^\infty_c \left(\R^{2N} \setminus \mathbb{D}^{2\mathrm{D}}\right), $$
where 
$$ \mathbb{D}^{2\mathrm{D}} \colonequals \left\{ (\mathbf{x}_1,\dots,\mathbf{x}_N) \in \mathbb{R}^{2N} : \; \mathbf{x}_j = \mathbf{x}_l  \text{ for some } j\neq l \right\}, $$
and then consider the unique self-adjoint extension $H^{\mathrm{2D}}_{\varepsilon}$ whose domain is included in the form closure of the above. Other self-adjoint extensions exist~\cite{CorOdd-18,CorFer-21} but shall not be considered hereafter. More details on these definitions will be provided in Section~\ref{sec:prelim} below. 

Our main results are stated as follows:

\begin{theorem}[\textbf{Relation between energies}] \label{thmEE}\mbox{}\\
    Let $\lambda^{\mathrm{2D}}_{k}$ and $\lambda^{1\mathrm{D}}_k$ be the $k$-th eigenvalues of $H^{\mathrm{2D}}_{\varepsilon}$ and $H^{1\mathrm{D}}$ respectively (counting the multiplicity), and let $e_{\varepsilon}$ be the ground energy of $H^{\mathrm{HO}}_{\varepsilon}$. For any fixed $k$, we have, in the $\varepsilon \to 0$ limit,
    \begin{equation*}
        \lambda^{\mathrm{2D}}_{k} =  {N}e_{\varepsilon} + \lambda^{\mathrm{1D}}_k + o(1).  \label{eqE}
    \end{equation*}
\end{theorem}

\begin{theorem}[\textbf{Relation between eigenfunctions}] \label{thmEF}\mbox{}\\
    With the same notation as in Theorem \ref{thmEE}, let $\{\Psi^{\mathrm{2D}}_{k}\}_k$ be an orthonormal basis of eigenfunctions corresponding to the eigenvalues $\{\lambda^{\mathrm{2D}}_{k}\}_k$. Then, after extracting a subsequence, there exists an orthonormal basis of eigenfunctions $ \{\psi^{\mathrm{1D}}_{k}\}_k$ corresponding to eigenvalues $\{\lambda^{1\mathrm{D}}_k\}_k$ such that, in the $\varepsilon \to 0$ limit, for any fixed $k$,
    $$\Psi^{\mathrm{2D}}_{k} - \psi^{1\mathrm{D}}_k U_{\varepsilon} \to 0 \quad  \text{ strongly in }  L^2(\mathbb{R}^{2N})$$
    and 
    $$ \varphi_{k}^{\varepsilon} \to  \psi^{1\mathrm{D}}_k \quad  \text{ strongly in }  H^1(\mathbb{R}^{N}\backslash \mathbb{D}^{1\mathrm{D}}), $$
    where 
    \begin{equation} \label{defU}
        U_{\varepsilon} (y_1,\dots,y_N)  = \prod_{j=1}^N u_{\varepsilon}(y_j)
    \end{equation}
    with $u_\varepsilon$ as in~\eqref{defu} and 
    \begin{equation}\label{eq:proj1D}
    \varphi_{k}^{\varepsilon} (x_1,\dots,x_N)\colonequals     \int_{\mathbb{R}^N}{\Psi_{k}^{\mathrm{2D}}} (\mathbf{x}_1,\dots,\mathbf{x}_N)e^{\mathrm{i}\alpha S(\mathbf{x}_1,\dots,\mathbf{x}_N)} U_{\varepsilon}({y}_1,\dots,{y}_N)\differential y_1\cdots\differential y_N,
    \end{equation}
    \begin{align}
    S(\mathbf{x}_1,\dots,\mathbf{x}_N) = \sum_{1\le j < l \le N}  \arctan\frac{y_j-y_l}{x_j-x_l}.\label{defS}
    \end{align}
\end{theorem}

We emphasize the role of the phase factor 
 \begin{equation*} 
 e^{\mathrm{i}\alpha S(\mathbf{x}_1,\dots,\mathbf{x}_N)} = \prod_{j<k} e^{\mathrm{i}\alpha \arctan\frac{y_j-y_k}{x_j-x_k}}
 \end{equation*}
apparent in~\eqref{eq:proj1D}. Extracting it from the many-body wave-functions has the effect of gauging the long-range magnetic interactions away, explaining that our limit model does not depend on $\alpha$. 
The only remnant of the 2D interactions is the hard-core constraint upon particle encounters imposed in the domain of the 1D Hamiltonian. 
As we explain in more detail in~\cite{RouYan-23a}, failing to perform this singular gauge transformation would lead to a limit 1D model of Calogero-type
\begin{equation*}
H^{\mathrm{Cal}}_{\alpha} \colonequals  \sum\limits_{j=1}^N \left(-\partial_{{x}_j}^2 + |x_j| ^2 \right) + 2 \alpha^2 \sum\limits_{1\le j <l\le N} \frac{1}{|{x}_j-{x}_l|^2},
\end{equation*}
which is also a candidate for the description of 1D anyons. Indeed, the energy of an ansatz 
$$\psi(x_1,\ldots,x_N) U_{\varepsilon} (y_1,\ldots,y_N)$$
under $H^{2\mathrm{D}}_{\varepsilon} $ is almost the sum of $Ne_{\varepsilon}$ and the energy of $\psi $ under $H^{\mathrm{Cal}}_{\alpha} $ instead of $H^{1\mathrm{D}}$. 
The correct ansatz actually is of the form 
\begin{equation}\label{eq:ansatz}
\psi(x_1,\ldots,x_N) U_{\varepsilon} (y_1,\ldots,y_N) e^{-\mathrm{i}\alpha S (\mathbf{x}_1,\ldots,\mathbf{x}_N)} 
\end{equation}
with $S$ as in~\eqref{defS}, whose gradient is close to $\mathbf{A}$ and thus compensates its action on $\psi$ in the above ansatz. 
One may check that 
 \begin{equation}\label{eq:grad S}
\grad_{\mathbf{x}_j} S = \mathbf{A}_j  
 \end{equation}
when $ x_l \neq x_j$ for all $l\neq j$, with a distributional component on the 1D diagonals $\{\exists\,  j\neq l,\,  x_j = x_l\}$ that gives no contribution when combined with the vanishing of trial states on this set. Note that $S$ is invariant under the exchange of particle labels, so that~\eqref{eq:ansatz} does have the necessary bosonic symmetry, unlike the function obtained by changing ~\eqref{defS} into 
$$
\sum_{1\le j < l \le N}  \arg \left(\mathbf{x_j} - \mathbf{x_k}\right).
$$ 
The proof for Theorem \ref{thmEE} will be separated into proving two opposite inequalities in two separate sections, upper and lower bounds for $\lambda_k^{2\mathrm{D}}$. Theorem~\ref{thmEF} follows from these sharp energy bounds, as we prove in the last section. In essence we follow a $\Gamma$-convergence approach, which is an equivalent way of checking convergence of operators in resolvent sense~\cite{DalMaso-93}. 

\bigskip

\noindent\textbf{Acknowledgments.} 
Funding from the European Research Council (ERC) under the European Union's Horizon 2020 Research and Innovation Programme (Grant agreement CORFRONMAT No 758620) is gratefully acknowledged. We had interesting discussions with Douglas Lundholm, Michele Correggi, Per Moosavi and Nikolaj Thomas Zinner.

\section{Preliminaries}\label{sec:prelim}

We first collect a few general facts. The more technically crucial in the sequel is the Hardy inequalities for 2D anyons that we recall and/or improve below. Before this, we briefly recall the definition of the operators of interest as Friedrichs extensions of non-negative quadratic forms. We conclude the section with standard definitions for the spectral data we are interested in, and by recalling elements of Girardeau's exact solution for the impenetrable 1D Bose gas~\cite{Girardeau-60,MinVig-22,MisEtalZin-22}.

\subsection{Domain of the Hamiltonians}

We start by defining the domains on which the operators $H^{2\mathrm{D}}_{\varepsilon} $ and $H^{ 1\mathrm{D} }$ will be defined hereafter. Throughout the paper, we use Friedrichs extensions to define the Hamiltonians in a weak (distributional) sense. For a more general discussion on that topic, readers can refer to, for example, \cite[Chapters 2 and 3]{Lewin-livre22} or~\cite{ReeSim1,ReeSim2}.

\subsubsection{2D model} We first discuss $H^{2\mathrm{D}}_{\varepsilon}$. 
As an operator on $L^2_{\mathrm{sym}}(\mathbb{R}^{2N})$, it can clearly be defined on the subspace 
$$
\mathcal{C}^{2\mathrm{D}} \colonequals \mathcal{C}^{\infty}_{c} (\mathbb{R}^{2N}\backslash \mathbb{D}^{2\mathrm{D}})\cap L^2_{\mathrm{sym}}(\mathbb{R}^{2N}),
$$
where $\mathbb{D}^{2\mathrm{D}} $ represents the diagonals in $\mathbb{R}^{2N}$, i.e. 
$$ 
\mathbb{D}^{2\mathrm{D}} \colonequals \left\{ (\mathbf{x}_1,\dots,\mathbf{x}_N) \in \mathbb{R}^{2N} : \mathbf{x}_j = \mathbf{x}_l  \text{ for some } j\neq l \right\}
$$
and $\mathcal{C}^{\infty}_{c} (\mathbb{R}^{2N}\backslash \mathbb{D}^{2\mathrm{D}}) $ is the set of smooth functions with compact support contained in $\mathbb{R}^{2N}\backslash \mathbb{D}^{2\mathrm{D}}$. 
Our operator is symmetric and corresponds to the quadratic form $\mathcal{E}_{\varepsilon}^{2\mathrm{D}}$ defined by 
\begin{equation*}
    \mathcal{E}^{2\mathrm{D}}_{\varepsilon}(\phi) \colonequals \langle \phi | H^{2\mathrm{D}}_{\varepsilon} \phi \rangle_{L^2} =  \sum_{j=1}^N  \int_{\mathbb{R}^{2N}}   \left( {  |{D_j\phi}|^2 +  V_{\varepsilon}(\mathbf{x}_j)\absolutevalue{\phi}^2 } \right), \quad \phi \in \mathcal{C}^{2\mathrm{D}}. 
\end{equation*}
Since the above is non-negative, i.e. $\mathcal{E}_{\varepsilon}^{2\mathrm{D}}(\phi) \ge 0 $ for all $\phi \in\mathcal{C}^{2\mathrm{D}}$, it is closable and we denote its closure by $( \mathcal{E}_{\varepsilon}^{2\mathrm{D}}, \mathcal{D}_q(H_{\varepsilon}^{2\mathrm{D}}) ) $, where $ \mathcal{D}_q(H_{\varepsilon}^{2\mathrm{D}}) $ is the completion of $\mathcal{C}^{2\mathrm{D}}$ under the norm 
$$
\sqrt{ \mathcal{E}_{\varepsilon}^{2\mathrm{D}} (\cdot) } + \norm{\cdot}_{L^2} 
$$ 
and is a dense subspace of $L^2_{\mathrm{sym}}(\mathbb{R}^{2N}) $~\cite[Theorem~3.10]{Lewin-livre22}.
For simplicity, we denote by $ \mathcal{E}_{\varepsilon}^{2\mathrm{D}}$ the closure  in the sequel. 
Then there exists a unique self-adjoint extension of $H^{2\mathrm{D}}_{\varepsilon}$, called Friedrichs extension, with extended domain $\mathcal{D}(H^{2\mathrm{D}}_{\varepsilon})$ contained in $\mathcal{D}_q(H_{\varepsilon}^{2\mathrm{D}}) $ \cite[Corollary 3.17]{Lewin-livre22}.  
One can check that
\begin{equation*}
    \mathcal{D}_q(H^{2\mathrm{D}}_{\varepsilon}) = \left\{ \phi \in L^2_{\mathrm{sym}}(\mathbb{R}^{2N}) : \forall j, \; D_j \phi, \; |\mathbf{x}_j| \phi \in L^2(\mathbb{R}^{2N})  \right\} 
\end{equation*}
is independent of $\varepsilon$ \cite[Theorem 5]{LunSol-14}.

\subsubsection{1D model} Our approach to $H^{1\mathrm{D}}$ is similar. 
It can be defined naturally on the dense subspace
$$
\mathcal{C}^{1\mathrm{D}}\colonequals \mathcal{C}_c^{\infty}(\mathbb{R}^N\backslash \mathbb{D}^{1\mathrm{D}})\cap L^2_{\mathrm{sym}}(\mathbb{R}^N) 
$$ 
that is a nice enough subspace of symmetric functions vanishing on diagonals. 
Clearly, $(H^{1\mathrm{D}}, \mathcal{C}^{
1\mathrm{D}})$ is symmetric and its quadratic form is 
\begin{equation*}
    \mathcal{E}^{1\mathrm{D}}(\varphi) \colonequals \langle \varphi | H^{1\mathrm{D}} \varphi \rangle_{L^2} = \sum_{j=1}^N \int_{\mathbb{R}^{N}} \LR{ \left|\partial_{x_j} \varphi \right|^2  + |x_j|^2\left|\varphi\right|^2 }, \quad \varphi\in \mathcal{C}^{1\mathrm{D}},
\end{equation*}
which is also bounded below by $0$. Then its closure is defined on $\mathcal{D}_q(H^{1\mathrm{D}})$, the completion of $\mathcal{C}^{1\mathrm{D}}$ under the norm 
$$
\sqrt{ \mathcal{E}^{1\mathrm{D}}(\cdot) } + \norm{\cdot}_{L^2}
$$ i.e. 
\begin{equation*}
    \mathcal{D}_q(H^{1\mathrm{D}}) = \left\{ \varphi \in H^1_0(\mathbb{R}^N\backslash \mathbb{D}^{1\mathrm{D}}) \cap L^2_{\mathrm{sym}}(\mathbb{R}^N) : \forall j, x_j\varphi \in L^2\left(\mathbb{R}^N\right) \right\},
\end{equation*}
where $H^1_0(\mathbb{R}^N\backslash \mathbb{D}^{1\mathrm{D}})$ is the completion of $\mathcal{C}_c^{\infty}(\mathbb{R}^N\backslash \mathbb{D}^{1\mathrm{D}}) $ under the natural $H^1$-norm. 
The Friedrichs extension of $(H^{1\mathrm{D}} ,\mathcal{C}^{1\mathrm{D}} ) $ has domain $\mathcal{D}(H^{1\mathrm{D}}) $ contained in $\mathcal{D}_q(H^{1\mathrm{D}})$.
According to the trace theorem~\cite[Section 5.5]{Evans}, there exists a unique bounded linear operator 
\begin{equation}\label{traceop}
T: H^1(\mathbb{R}^{N} ) \to \bigoplus _{1\le j<l \le N } L^2\left(\left\{(x_1,\dots,x_N) \in \mathbb{R}^N: x_j=x_l \right\}\right)
\end{equation}
such that $ T \varphi = (\varphi|_{x_j=x_l})_{j<l} $ if $\varphi \in H^1(\mathbb{R}^N) \cap \mathcal{C}(\mathbb{R}^N) $.
A function $\varphi$ in $H^1(\mathbb{R}^N) $ has zero trace on the diagonals $\mathbb{D}^{1\mathrm{D}}$, i.e. $T\varphi =0$, if and only if it is in $H^1_0(\mathbb{R}^N\backslash \mathbb{D}^{1\mathrm{D}})$.

\subsection{Hardy inequalities for 2D anyons}

An important ingredient of our proof below is that the magnetic-gauge picture 2D Hamiltonian defined above controls the inverse square distance between any pair of particles. Such Hardy-type inequalities were investigated in~\cite{HofLapTid-08,LunSol-13a,LunSol-13b,LarLun-16,LunQva-20}, and we recall some of the known results here.

We first consider the case where two particles out of $N$ are isolated from the others. A Hardy inequality for any $\alpha \neq 0$ follows from~\cite{LunSol-13a} in this case. We next deal with the general case to obtain an inequality for any $\alpha \neq 0$ by combining the isolated case with estimates from~\cite{HofLapTid-08}.

\begin{theorem}[\textbf{Isolated two-anyon Hardy}]\label{isolatedHineq}\mbox{}\\
Let $E$ be an open subset of $\mathbb{R}^{2(N-1)}$, and let $\gamma \in (0,\infty]$ be a constant. 
Then for any $\phi \in L^2
_{sym}(\mathbb{R}^{2N})$ such that $|D_1 \phi| , |D_2 \phi| \in L^2(\mathbb{R}^{2N})$, we have
\begin{equation}\label{isolatedH}
\int_{E_{\gamma}} \LR{ \absolutevalue{D_1 \phi}^2 +\absolutevalue{D_2 \phi}^2 } \ge  C_{\alpha}  \int_{E_{\gamma}} \frac{\absolutevalue{\phi}^2}{\absolutevalue{\mathbf{x}_1-\mathbf{x}_2}^2}
\end{equation}
with 
\begin{multline*} \label{Ehdef}
    E_\gamma \colonequals \Big\{ \left(\mathbf{x}_1,\dots,\mathbf{x}_N\right) \in \mathbb{R}^{2N}: \left(\mathbf{x}_1+\mathbf{x}_2, \mathbf{x}_3,\dots,\mathbf{x}_N\right) \in E, \\
    \left|{\mathbf{x}_1-\mathbf{x}_2}\right| < \min \left\{ \gamma, \vert\mathbf{x}_1+\mathbf{x}_2 - 2 \mathbf{x}_3\vert, \dots, \vert\mathbf{x}_1+\mathbf{x}_2 - 2 \mathbf{x}_N\vert \right\}  \Big\}
\end{multline*}
and  
\begin{equation}\label{eq:Calpha}
C_{\alpha} \colonequals { 2 \min_{q\in\mathbb{Z}} \absolutevalue{\alpha-2q}^2 } \in (0,2). 
\end{equation}
\end{theorem}

\begin{proof}
By a density argument, we may work with a smooth $\phi$ without loss of generality. The following proof imitates that of \cite[Theorem~4]{LunSol-13a}.

We parameterize the variables $\mathbf{x}_1$ and $\mathbf{x}_2$ by the center-of-mass coordinate 
$$
\mathbf{R}\colonequals \frac{\mathbf{x}_1+\mathbf{x}_2}{2}
$$ 
and the relative coordinate 
$$
\mathbf{r}\colonequals \mathbf{x}_1-\mathbf{x}_2.
$$
A direct calculation yields
\begin{equation*}
    \LR{D_1-D_2}\phi =2 \LR{-\mathrm{i}\grad_{\mathbf{r}} +  \mathbf{a}(\mathbf{r};\mathbf{R},\mathbf{x}_3,\dots,\mathbf{x}_N) } \phi,
\end{equation*}
where the vector potential 
$$\mathbf{a} (\mathbf{r};\mathbf{R},\mathbf{x}_3,\dots,\mathbf{x}_N) \colonequals \alpha\frac{\mathbf{r}^{\perp}}{\absolutevalue{\mathbf{r}}^2} + \alpha \sum_{k=3}^N \LR{\frac{\LR{\mathbf{r} + 2(\mathbf{R}-\mathbf{x}_k)}^{\perp}}{\absolutevalue{\mathbf{r} + 2(\mathbf{R}-\mathbf{x}_k)}^2} + \frac{\LR{\mathbf{r} - 2(\mathbf{R}-\mathbf{x}_k)}^{\perp}}{\absolutevalue{\mathbf{r} - 2(\mathbf{R}-\mathbf{x}_k)}^2}}$$ is antipodal-antisymmetric with respect to $\mathbf{r}$.
By definition of $E_{\gamma}$, we have
$$ \abs{\mathbf{r}} < 2 \abs{\mathbf{R}-\mathbf{x}_k}, \quad \forall k\ge 3. $$
Hence, $$\grad_{\mathbf{r}} \wedge \mathbf{a} = 0 \quad \text{on} \quad  E_{\gamma}\backslash\{ \mathbf{0} \} $$ and $$\int_{\Gamma}\mathbf{a}\cdot \differential \mathbf{r} = 2\pi \alpha $$
for any simple loop $\Gamma$ in $E_{\gamma}\backslash\{ \mathbf{0}\}$ enclosing $\{ \mathbf{0} \}$. 
Applying \cite[Lemma~2]{LunSol-13a} to $\phi$, we obtain
\begin{equation}\label{afterLemma2}
    \int_{\Omega(\mathbf{R}, \mathbf{x}_3,\dots,\mathbf{x}_N) } \absolutevalue{\LR{D_1-D_2}\phi}^2 \differential \mathbf{r} \ge 2 C_{\alpha}\int_{\Omega(\mathbf{R}, \mathbf{x}_3,\dots,\mathbf{x}_N) } \frac{\absolutevalue{\phi}^2 }{\absolutevalue{\mathbf{r}}^2}\differential \mathbf{r},
\end{equation}
where 
$$\Omega(\mathbf{R}, \mathbf{x}_3,\dots,\mathbf{x}_N) \colonequals \left\{ \mathbf{r} \in \mathbb{R}^2 : \abs{\mathbf{r}} \in \left(0, \min \{ \gamma, 2 \vert\mathbf{R} - \mathbf{x}_3\vert, \dots, 2 \vert\mathbf{R} - \mathbf{x}_N\vert \} \right) \right\}. $$ 
Integrating the estimate~\eqref{afterLemma2} over $E$ on both sides gives
\begin{equation*}\label{Ghardyeq1}
    \int_{E_{\gamma}} \absolutevalue{\LR{D_1-D_2}\phi}^2  \ge 2 C_{\alpha}\int_{E_{\gamma}} \frac{\absolutevalue{\phi}^2 }{|{\mathbf{x}_1-\mathbf{x}_2}|^2}.
\end{equation*}
Using the fact that $$2\left(\absolutevalue{D_1 \phi}^2 +\absolutevalue{D_2 \phi}^2\right) = \absolutevalue{\LR{D_1-D_2}\phi}^2 + \absolutevalue{\LR{D_1+D_2}\phi}^2,$$ 
we finally get the isolated Hardy inequality~\eqref{isolatedH}.
\end{proof}

To estimate the integral on the complementary domain of the isolated two-anyon domain $E_{\gamma}$ defined in Theorem \ref{isolatedHineq}, we need the following lemma, which follows immediately from~\cite[Lemma~3.6]{HofLapTid-08}:

\begin{lemma}[\textbf{Three-particle Hardy}]\label{3Hardy}\mbox{}\\
Let $\phi \in H^1(\mathbb{R}^{2N})$, then for any integer $k\in [3,N]$, we have 
\begin{equation*}
    \int_{\mathbb{R}^{2N}} \LR{ \absolutevalue{\grad_{\mathbf{x}_1}{\phi}}^2 + \absolutevalue{\grad_{\mathbf{x}_2}{\phi}}^2 + \absolutevalue{\grad_{\mathbf{x}_k}{\phi}}^2 } \ge 3 \int_{\mathbb{R}^{2N}}   \frac{\absolutevalue{\phi}^2}{\rho_k^2},
\end{equation*}
where $$\rho_k^2 \colonequals \absolutevalue{\mathbf{x}_1 - \mathbf{x}_2}^2 + \absolutevalue{\mathbf{x}_1 - \mathbf{x}_k}^2 + \absolutevalue{\mathbf{x}_2 - \mathbf{x}_k}^2 .$$
\end{lemma}

\begin{remark}[Three-particle Hardy for anyons]\mbox{}\\
With the help of the diamagnetic inequality from~\cite[Lemma 4]{LunSol-14}, $ |\phi|$ is in $H^1(\mathbb{R}^{2N})$ provided that $\phi \in L^2_{\mathrm{sym}}(\mathbb{R}^{2N})$ and $\forall j, \, |D_j \phi| \in L^2(\mathbb{R}^{2N})$. 
Applying Lemma \ref{3Hardy} above to such $\abs{\phi}$ and using the diamagnetic inequality, yields
\begin{equation}
    \int_{\mathbb{R}^{2N}} \LR{\absolutevalue{D_1\phi}^2 + \absolutevalue{D_2\phi}^2 +\absolutevalue{D_k \phi}^2 } \ge 3 \int_{\mathbb{R}^{2N}}   \frac{\absolutevalue{\phi}^2}{\rho_k^2}  . \label{est3body}
\end{equation}
\end{remark}

The above gives a control on three-particle encounters that we may combine with Theorem~\ref{isolatedHineq} to prove a many-anyons Hardy inequality on the whole configuration space with fewer restrictions on $\alpha$ than in~\cite[Theorem~4]{LunSol-13a}.

\begin{theorem}[\textbf{Many-anyon Hardy for $\alpha \in (0,2)$}] \label{0to2hardy}\mbox{}\\
Let $\phi \in L^2_{\mathrm{sym}}(\mathbb{R}^{2N})$ such that $\forall j, \, |D_j \phi| \in L^2(\mathbb{R}^{2N})$, then we have
\begin{equation*}
    \sum_{k=1}^N \int_{\mathbb{R}^{2N}} \absolutevalue{D_k\phi}^2 \ge \frac{2C_{\alpha}}{(N-1)\LR{2+3(N-2)C_{\alpha}}} \sum_{1\le j <l \le N}\int_{\mathbb{R}^{2N}} \frac{\absolutevalue{\phi}^2}{|\mathbf{x}_j - \mathbf{x}_l|^2}. \label{noRalpha}
\end{equation*}
with $C_\alpha$ as in~\eqref{eq:Calpha}.
\end{theorem}

Observe that the constant in the above scales as $N^{-2}$ for large $N$, which is worst than the optimal $N^{-1}$ obtained in~\cite[Theorem~4]{LunSol-13a} for special values of $\alpha$. We shall work at a fixed $N$ in the sequel.

\begin{proof}
With the same notation as in Theorem \ref{isolatedHineq}, we let $\gamma$ be $\infty$ and let $E$ be $\mathbb{R}^{2(N-1)}$.
It is clear that $$\mathbb{R}^{2N} \backslash E_{\infty} = \bigcup_{k=3}^N F_k, $$  
where
$$ F_k \colonequals \left\{(\mathbf{x}_1,\dots,\mathbf{x}_N) \in \mathbb{R}^{2N}: \absolutevalue{\mathbf{x}_1 - \mathbf{x}_2} \ge \absolutevalue{\mathbf{x}_1+\mathbf{x}_2 - 2 \mathbf{x}_k}  \right\}. $$
Hence, for $(\mathbf{x}_1,\dots,\mathbf{x}_N) \in F_k $, we have $$\absolutevalue{\mathbf{x}_1 -\mathbf{x}_2 }  \ge \max\left\{\absolutevalue{\mathbf{x}_1 -\mathbf{x}_k }, \absolutevalue{\mathbf{x}_2 -\mathbf{x}_k }\right\} ,$$ 
and then 
$$\rho^2_k \le 3\absolutevalue{\mathbf{x}_1 -\mathbf{x}_2 }^2,$$ 
where $\rho_k^2$ is defined in Lemma \ref{3Hardy}. 
Therefore, with the help of the estimate (\ref{est3body}), it turns out that
\begin{multline}\label{after2.1}
    \sum_{k=3}^N \int_{\mathbb{R}^{2N}} \LR{\absolutevalue{D_1\phi}^2 + \absolutevalue{D_2\phi}^2 +\absolutevalue{D_k \phi}^2 }\ge  \sum_{k=3}^N  3 \int_{\mathbb{R}^{2N}}   \frac{\absolutevalue{\phi}^2}{\rho_k^2}
    \ge \sum_{k=3}^N 3\int_{F_k}\frac{\absolutevalue{\phi}^2}{\rho_k^2}\\
    \ge \sum_{k=3}^N \int_{F_k}\frac{\absolutevalue{\phi}^2}{|{\mathbf{x}_1 - \mathbf{x}_2}|^2} 
    \ge \int_{\mathbb{R}^{2N}\backslash E_{\infty}} \frac{\absolutevalue{\phi}^2}{|{\mathbf{x}_1 - \mathbf{x}_2}|^2} .
\end{multline}
Combining (\ref{isolatedH}) with (\ref{after2.1}), we obtain 
\begin{equation}\label{for12}
   \LR{N-2+C_{\alpha}^{-1}} \int_{\mathbb{R}^{2N}} \LR{ \absolutevalue{D_1\phi}^2 + \absolutevalue{D_2\phi}^2 }  +   { \sum_{k=3}^N \int_{\mathbb{R}^{2N}} \absolutevalue{D_k\phi}^2} \ge \int_{\mathbb{R}^{2N}} \frac{\absolutevalue{\phi}^2}{|{\mathbf{x}_1 - \mathbf{x}_2}|^2}. 
\end{equation}
Since $\phi$ is symmetric, the estimate (\ref{for12}) above implies
\begin{equation*}\label{forjl}
   \LR{N-2+C_{\alpha}^{-1}} \int_{\mathbb{R}^{2N}} \LR{ \absolutevalue{D_j\phi}^2 + \absolutevalue{D_l\phi}^2 }  +   {\sum_{1\le k \le N,\, k\neq j,l} \int_{\mathbb{R}^{2N}} \absolutevalue{D_k\phi}^2} \ge \int_{\mathbb{R}^{2N}} \frac{\absolutevalue{\phi}^2}{|\mathbf{x}_j - \mathbf{x}_l|^2}, \quad j \neq l.
\end{equation*}
Summing over all possible pairs, we conclude the proof of Theorem \ref{0to2hardy}.
\end{proof}

One application of the above is to clarify the quadratic domain of the 2D magnetic-gauge operator: it must be contained in the usual $H^1$ space.

\begin{cor}[\textbf{$H^1$-regularity}]\label{cptD}\mbox{}\\
Consider the quadratic form domain of the harmonic oscillator
\begin{equation*}
    H^1_V (\mathbb{R}^{2N}) \colonequals \left\{ \phi \in H^1(\mathbb{R}^{2N}) : \forall j , |{\mathbf{x}_j}|\phi \in L^2(\mathbb{R}^{2N}) \right\}.
\end{equation*}
Then the domain of the quadratic form of $H^{2\mathrm{D}}_{\varepsilon} $, $\mathcal{D}_q(H^{2\mathrm{D}}_{\varepsilon}) $, is contained in $ H^1_{V_{}}(\mathbb{R}^{2N})$.
\end{cor}

\begin{proof}
    Let $\phi$ be in $\mathcal{D}_q(H^{2\mathrm{D}}_{\varepsilon})  $, then $\phi \in L^2_{\mathrm{sym}}(\mathbb{R}^{2N})$ and $|D_j \phi| \in L^2(\mathbb{R}^{2N})$ for all $j$. 
    Using the Cauchy-Schwarz inequality and applying Theorem \ref{0to2hardy} to $\phi$, we have
    \begin{equation*}
        \int_{\mathbb{R}^{2N}} \abs{\mathbf{A}_j\phi}^2 \le C_N \sum_{1\le k \le N,\, k\neq j}  \int_{\mathbb{R}^{2N}} \frac{\absolutevalue{\phi}^2}{|\mathbf{x}_k-\mathbf{x}_j|^2}< \infty,
    \end{equation*}
    which implies that $|{\mathbf{A}_j}\phi|\in L^2(\mathbb{R}^{2N}) $. 
    Since $|D_j \phi| $ is also in $L^2(\mathbb{R}^{2N})  $, $$\left|\grad_{\bx_j}\phi \right| = \left| \mathrm{i}\LR{D_j \phi - \alpha \mathbf{A}_j \phi} \right| $$
    is clearly in $L^2(\mathbb{R}^{2N}) $ for all $j$,  which concludes the proof.
\end{proof}

\subsection{Eigenvalues and eigenfunctions}

\subsubsection{1D model}

The Tonks-Girardeau gas is exactly soluble via Bose/Fermi mapping~\cite{Girardeau-60,MinVig-22,MisEtalZin-22}. This means that there is a one-to-one correspondence between its eigenfunctions and those of the free Fermi gas (Slater determinants). Any eigenfunction $\psione_k$ of $H^{1\mathrm{D}}$ is of the form
$$ \psi_k^{1\mathrm{D}} (x_1,\dots,x_N) = c_k \prod_{i<j} \mathrm{sgn} (x_i-x_j) \underset{i,j}{\det} (v_i (x_j)),$$
where $v_1,\ldots,v_N$ are eigenfunctions of the one-particle harmonic oscillator
$$-\partial_x + x^2.$$
The corresponding eigenvalue $\lone_k$ is the sum of the one-particle eigenvalues corresponding to $v_1,\ldots,v_N$.
With this correspondence, we know that the eigenfunctions of $H^{1\mathrm{D}} $ are all of the form 
\begin{equation*}\label{defEF1D}
    \psi^{1\mathrm{D}}_k(x_1,\dots,x_N) = p_k(x_1,\dots,x_N) \prod_{1\le j < l \le N}\absolutevalue{x_j-x_l}e^{-(x_1^2+\cdots+x_N^2)/2}
\end{equation*}
for some polynomial $p_k$, and that they form an orthonormal basis of $L^2(\mathbb{R}^N)$.

\subsubsection{2D model}

As regards the eigenfunctions of $H^{2\mathrm{D}}_{\varepsilon} $,  we prove that

\begin{prop}[\textbf{Diagonalization of $H^{2\mathrm{D}}_{\varepsilon}$ }]\mbox{}\\
    There exists an orthonormal basis $(\Psi_k^{2\mathrm{D}})_{k\in\mathbb{N}^*}$ of $L^2_{\mathrm{sym}}(\mathbb{R}^{2N})$, and a non-decreasing divergent sequence $(\lambda^{2\mathrm{D}}_k)_{k\in\mathbb{N}^*}$ such that 
    \begin{equation*}
        H^{2\mathrm{D}}_{\varepsilon} \Psi_k^{2\mathrm{D}} = \lambda^{2\mathrm{D}}_k \Psi_k^{2\mathrm{D}}.
    \end{equation*}
\end{prop}

\begin{proof}
    Since the quadratic form $\mathcal{E}^{2\mathrm{D}}$ is non-negative,  
    $$
    H^{2\mathrm{D}}_{\varepsilon}+1: \mathcal{D}(H^{2\mathrm{D}}_{\varepsilon}) \to L^2_{\mathrm{sym}}(\mathbb{R}^{2N})
    $$ is invertible and 
    $$(H^{2\mathrm{D}}_{\varepsilon}+1)^{-1}: L^2_{\mathrm{sym}}(\mathbb{R}^{2N}) \to \mathcal{D}(H^{2\mathrm{D}}_{\varepsilon}) \subset L^2_{\mathrm{sym}}(\mathbb{R}^{2N}) 
    $$ is bounded. 
    With the help of Corollary~\ref{cptD} above and the fact that $H^1_V(\mathbb{R}^{2N})$ is compactly embedded in $ L^2(\mathbb{R}^{2N}) $, we deduce that $(H^{2\mathrm{D}}_{\varepsilon}+1)^{-1}$ is a compact operator. 
    Applying the spectral theorem~\cite[Chapters~4 and~5]{Lewin-livre22}, there exists an orthonormal basis $(\Psi_k^{2\mathrm{D}})_{k\in\mathbb{N}^*}$ of $L^2_{\mathrm{sym}}(\mathbb{R}^{2N})$, and a  decreasing vanishing sequence $(\gamma^{}_k)_{k\in\mathbb{N}^*}$ such that 
    \begin{equation*}
        (H^{2\mathrm{D}}_{\varepsilon}+1)^{-1} \Psi_k^{2\mathrm{D}} = \gamma^{}_k \Psi_k^{2\mathrm{D}}.
    \end{equation*}
    This implies that such $ \Psi^{2\mathrm{D}}_k $'s are all in $\mathcal{D}(H^{2\mathrm{D}}_{\varepsilon})$ and 
    \begin{equation*}
        H^{2\mathrm{D}}_{\varepsilon}\Psi_k^{2\mathrm{D}} = (\gamma^{-1}_k -1 ) \Psi_k^{2\mathrm{D}},
    \end{equation*}
    which concludes the proof.
\end{proof}

We shall characterize eigenvalues by standard min-max formulae~\cite[Section 5.5]{Lewin-livre22}:

\begin{theorem}[\textbf{Courant-Fischer min-max formulae}]\label{minmax}\mbox{}\\
Let $(H,\mathcal{D}(H))$ be a bounded-below self-adjoint operator, and let $(\mathcal{E},\mathcal{D}_q(H))$ be its corresponding closed quadratic form. Then
$$\mu_{k} \colonequals \inf_{\substack{\text{subspace }V \subset \mathcal{D}(H) \\ \mathrm{dim} V = k}}  \sup_{\substack{\Psi \in V,\\\norm{\Psi} = 1}} \left\langle\Psi| H \Psi  \right\rangle  = \inf_{\substack{\text{subspace }V \subset \mathcal{D}_q(H) \\ \mathrm{dim} V = k}}  \sup_{\substack{\Psi \in V,\\\norm{\Psi} = 1}} \mathcal{E}(\Psi)  $$
is the $k$-th eigenvalue of $H$ counted with multiplicity or is $\inf \sigma_{ess}(A)$, the bottom of the essential spectrum of $H$.
\end{theorem}

Thanks to the diagonalization property of $H^{2\mathrm{D}}_{\varepsilon} $ and $H^{1\mathrm{D}} $, we know that they have empty essential spectrum. Hence, their eigenvalues can be characterized by the min-max formulae above.

\section{Energy upper bounds}

In this section, we bound the 2D eigenvalues from above by constructing appropriate trial states:

\begin{theorem}[\textbf{Upper bounds for 2D eigenvalues}]\label{thmUpperB}\mbox{}\\
Let $\lambda^{\mathrm{2D}}_{k}$ and $\lambda^{1\mathrm{D}}_k$ be the $k$-th eigenvalues of $H^{\mathrm{2D}}_{\varepsilon}$ and $H^{1\mathrm{D}}$ respectively (counting the multiplicity), and let $e_{\varepsilon}$ be the ground energy of $H^{\mathrm{HO}}_{\varepsilon}$. Then we have the following inequality:
\begin{equation*} \label{upperineq}
    {\lambda^{2\mathrm{D}}_{k} \le Ne_{\varepsilon}} + \lambda^{1\mathrm{D}}_{k}.
\end{equation*}
\end{theorem}

\begin{proof}
    According to the min-max formulae in Theorem~\ref{minmax}, the energy $\lambda^{2\mathrm{D}}_{k}$ can be written as 
\begin{equation}\label{eq2Denergy}
    \lambda^{2\mathrm{D}}_k = \inf_{\substack{\text{linearly independent}\\\Psi_1, \dots,\Psi_k \in \mathcal{D}_q(H^{2\mathrm{D}}_{\varepsilon})}}  \sup_{\substack{\norm{\Psi}_{L^2(\mathbb{R}^{2N})}=1\\ \Psi\in\operatorname{span}\{\Psi_1, \dots,\Psi_k\}}} \mathcal{E}^{2\mathrm{D}}_{\varepsilon} (\Psi).
\end{equation} 
We hence obtain an upper bound for $\lambda^{2\mathrm{D}}_k$ by constructing trial states of the form 
\begin{equation}\label{eq:formansatz}
\Psi(\mathbf{x}_1,\dots,\mathbf{x}_N)=\psi({x}_1,\dots,{x}_N)U_{\varepsilon}({y}_1,\dots,{y}_N)e^{-\mathrm{i}\alpha S(\mathbf{x}_1,\dots,\mathbf{x}_N)}
\end{equation}
with $\psi \in \mathcal{D}_q(H^{1\mathrm{D}}_{})$, $U_{\varepsilon}$ and $S$ as in~\eqref{defU} and~\eqref{defS} respectively. 

Notice that $S$ is smooth away from the $x$-direction diagonals, i.e. $S$ is well-defined and differentiable on 
\begin{equation}\label{lam0}
    \Lambda_0 \colonequals \left\{ (\mathbf{x}_1,\dots,\mathbf{x}_N) \in \mathbb{R}^{2N}, |x_j - x_k| > 0, \forall 1\le j<k\le N \right\}.
\end{equation}  
One can check that $S$ has a gap on the $x$-direction diagonals, hence, $S$ fails to be in $H^1_{\rm loc}(\mathbb{R}^{2N})$, but $S$ is however in $H^1_{\rm loc}(\Lambda_0)$.
A direct calculation based on~\eqref{eq:grad S} yields 
\begin{equation} \label{+phase}
\begin{split}
    { -\mathrm{i}\grad_{\mathbf{x}_j} } \LR{\psi U_{\varepsilon}}  = { -\mathrm{i}\grad_{\mathbf{x}_j} } \LR{\Psi e^{\mathrm{i}\alpha S }} & =  -\mathrm{i} \LR{\grad_{\mathbf{x}_j} \Psi}e^{\mathrm{i}\alpha S }  + \alpha \LR{\grad_{\bx_{j}}S }\Psi e^{\mathrm{i}\alpha S } \\
    & = \LR{ \LR{-\mathrm{i}\grad_{\mathbf{x}_j}+ \alpha {\mathbf{A}_j}}\Psi} e^{\mathrm{i}\alpha S }  = \LR{D_j\Psi} e^{\mathrm{i}\alpha S}
\end{split}
\end{equation} 
in the sense of distributions in $ \mathcal{D}'(\Lambda_0) $. Hence $D_j \Psi \in L^2(\Lambda_0)$. But $\psi \in \mathcal{D}_q(H^{1\mathrm{D}})$ vanishes on the diagonals and thus $\Psi$ is continuous across the components of $\Lambda_0$ despite the jumps in its phase. It follows that (\ref{+phase}) also holds around the singularity of $S$, i.e. can be extended to $\mathcal{D}'(\mathbb{R}^{2N}) $.
Hence, $D_j \Psi $ is in $L^2(\mathbb{R}^{2N})$, and then the trial state $\Psi$ is indeed in $\mathcal{D}_q(H^{2\mathrm{D}}_{\varepsilon})$ with the energy relation: 
\begin{equation}\label{eqDiff}
\begin{split}
    \mathcal{E}^{2\mathrm{D}}_{\varepsilon} ( \Psi ) & =  \sum_{j=1}^N  \int_{\mathbb{R}^{2N}}   \left( {  \left|{D_j\Psi}\right|^2 +  V_{\varepsilon}(\mathbf{x}_j)\absolutevalue{\Psi}^2 } \right) \\
    & =  \sum_{j=1}^N  \int_{\mathbb{R}^{2N}}  \LR{  \left| \partial_{x_j}\psi \right|^2  +  |{x_j}|^2 \absolutevalue{\psi}^2}  \abs{U_{\varepsilon}}^2 +   \sum_{j=1}^N  \int_{\mathbb{R}^{2N}} \LR{ \left| \partial_{y_j}U_{\varepsilon} \right|^2  + \frac{1}{\varepsilon^2} |{y_j}|^2 |{U_{\varepsilon}}|^2 }\absolutevalue{\psi}^2\\
    & = \mathcal{E}^{1\mathrm{D}} ( \psi ) + Ne_{\varepsilon} .
\end{split}
\end{equation}

Take now $\psi_1^{1\mathrm{D}}, \dots, \psi_k^{1\mathrm{D}}$ to be orthonormal eigenfunctions for the  1D eigenvalues $\lambda^{1\mathrm{D}}_1,\dots,\lambda^{1\mathrm{D}}_k $ respectively. The corresponding trial states as in~\eqref{eq:formansatz} are clearly orthonormal in $L^2 (\R^{2N})$. 
Using (\ref{eq2Denergy}) and (\ref{eqDiff}), we obtain
\begin{equation*}
\begin{split}
\lambda^{2\mathrm{D}}_k & \le   
\sup_{\substack{\Psi = \psi U_{\varepsilon}e^{-\mathrm{i}\alpha S}, \\  \psi\in\operatorname{span}\{\psi_1^{1\mathrm{D}}, \dots,\psi_k^{1\mathrm{D}}\}}} \mathcal{E}^{2\mathrm{D}}_{\varepsilon}(\Psi),\label{Rquotient}\\
& = Ne_{\varepsilon} + 
\sup_{{ \psi\in\operatorname{span}\{\psi_1^{1\mathrm{D}}, \dots,\psi_k^{1\mathrm{D}}\}}} \mathcal{E}^{1\mathrm{D}}(\psi)\\
& = N e_{\varepsilon} + \lambda_k^{1\mathrm{D}},
\end{split}
\end{equation*} 
where the suprema are over $L^2$-normalized functions. 
This concludes the proof of Theorem~\ref{thmUpperB}.
\end{proof}

\section{Energy lower bounds}

In this section, we will conclude the proof of Theorem~\ref{thmEE} by proving the lower bounds matching the upper bounds of Theorem \ref{thmUpperB}.

\begin{theorem}[\textbf{Lower bounds for 2D eigenvalues}]\label{thmLowerB}\mbox{}\\
Let $\lambda^{\mathrm{2D}}_{k}$ and $\lambda^{1\mathrm{D}}_k$ be the $k$-th eigenvalues of $H^{\mathrm{2D}}_{\varepsilon}$ and $H^{1\mathrm{D}}$ respectively (counting the multiplicity), and let $e_{\varepsilon}$ be the ground energy of $H^{\mathrm{HO}}_{\varepsilon}$. Then 
\begin{equation*}
    \liminf_{\varepsilon\to 0}\LR{\lambda^{2\mathrm{D}}_{k} - Ne_{\varepsilon}} \ge \lambda^{1\mathrm{D}}_{k}.
\end{equation*}
\end{theorem}

Let $\Psi^{\mathrm{2D}}_{ k}$ be an eigenfunction of the 2D Hamiltonian, corresponding to the eigenvalue $\lambda^{2\mathrm{D}}_{k} $. 
Recall that $U_{\varepsilon} $ defined in (\ref{defU}) is strictly positive, so that we can rewrite $\Psi^{\mathrm{2D}}_{k} $ as 
$$\Psi^{\mathrm{2D}}_{ k} (\mathbf{x}_1,\dots,\mathbf{x}_N) =  \Phi_{k} (\mathbf{x}_1,\dots,\mathbf{x}_N)  U_{\varepsilon} ({y}_1,\dots,{y}_N).$$
Then we can express the energy $\lambda^{\mathrm{2D}}_{k} $ in terms of $\Phi_{k}$ and $U_{\varepsilon} $ as follows:

\begin{prop}[\textbf{Energy decoupling}] \label{decoupE}\mbox{}\\
With the same notation as above, the following identity holds:
\begin{equation*}
\lambda^{\mathrm{2D}}_{k} = N  e_{\varepsilon} 
+ \sum\limits_{j=1}^N \int_{\mathbb{R}^{2N}} \left|D_j\Phi_{k}\right|^2 U^2_{\varepsilon} 
+ \sum\limits_{j=1}^N \int_{\mathbb{R}^{2N}} |x_j|^2 \left| \Psi^{\mathrm{2D}}_{k}\right|^2.
\end{equation*}
\end{prop}

\begin{proof}
Using integration by parts and the fact that 
$$ \sum_{j=1}^N \int_{\mathbb{R}^{2N}} \absolutevalue{y_j}^2 \absolutevalue{\Psi^{\mathrm{2D}}_{k}}^2 < \infty,$$
we obtain 
\begin{equation}\label{proofED1}
\int_{\mathbb{R}^{2N}} \absolutevalue{\grad_{\mathbf{x}_j}\Psi^{\mathrm{2D}}_{k}}^2 = \int_{\mathbb{R}^{2N}} \absolutevalue{\Phi_{k}}^2 U_{\varepsilon} \LR{-\partial_{y_j}^2U_{\varepsilon}} + \int_{\mathbb{R}^{2N}}\absolutevalue{\grad_{\mathbf{x}_j}\Phi_{k}}^2 U^2_{\varepsilon}.
\end{equation}
By definition, we have 
$$
-\partial_{y_j}^2 U_{\varepsilon} + \varepsilon^{-2} |y_j|^2 U_{\varepsilon} = e_{\varepsilon} U_{\varepsilon}
$$ 
for each $j$. Inserting this in Equation (\ref{proofED1}) we obtain
\begin{equation}\label{proofDecoupE1}
\int_{\mathbb{R}^{2N}} \LR{\absolutevalue{\grad_{\mathbf{x}_j}\Psi^{\mathrm{2D}}_{k}}^2 +V_{\varepsilon}(\mathbf{x}_j)\absolutevalue{\Psi_{k}^{\mathrm{2D}}}^2 } = e_{\varepsilon} + \int_{\mathbb{R}^{2N}}\absolutevalue{\grad_{\mathbf{x}_j}\Phi_{k}}^2 U^2_{\varepsilon} + \int_{\mathbb{R}^{2N}} |x_j|^2 \absolutevalue{\Psi^{\mathrm{2D}}_{k}}^2.
\end{equation}
On the other hand
\begin{equation*}
\begin{split}
\absolutevalue{D_j{\Psi^{\mathrm{2D}}_{k}}}^2 
& = \absolutevalue{\grad_{\mathbf{x}_j} \Psi^{\mathrm{2D}}_{k} }^2 + \mathbf{J}_j( \Psi^{\mathrm{2D}}_{k}) \cdot \mathbf{A}_j(\mathbf{x}_j) + \absolutevalue{ \Psi^{\mathrm{2D}}_{k} } \absolutevalue{\mathbf{A}_j(\mathbf{x}_j)}^2,  
\end{split}
\end{equation*}
where 
$$
\mathbf{J}_j( \Psi ) \colonequals \mathrm{i} \LR{\Psi \grad_{\mathbf{x}_j} \overline{\Psi} - \overline{\Psi} \grad_{\mathbf{x}_j} \Psi}.
$$
And similarly, \begin{equation*}
\absolutevalue{D_j{\Phi_{k}}}^2 = \absolutevalue{\grad_{\mathbf{x}_j}\Phi_{k}}^2 +  \mathbf{J}_j( \Phi_{k}) \cdot \mathbf{A}_j(\mathbf{x}_j) + \absolutevalue{\Phi_{k}}^2 \absolutevalue{\mathbf{A}_j(\mathbf{x}_j)}^2.
\end{equation*}
We observe that 
$$
\mathbf{J}_j( \Psi^{\mathrm{2D}}_{k} ) = U^2_{\varepsilon} \mathbf{J}_j( \Phi_{k} ),
$$
and hence
\begin{equation}\label{proofDecoupE2}
\absolutevalue{D_j{\Psi^{\mathrm{2D}}_{k}}}^2 = \absolutevalue{\grad_{\mathbf{x}_j}\Psi_{k}^{\mathrm{2D}}}^2 + \absolutevalue{D_j{\Phi_{k}}}^2 U^2_{\varepsilon} - \absolutevalue{\grad_{\mathbf{x}_j}\Phi_{k}}^2 U^2_{\varepsilon}. 
\end{equation}
Combining Equations (\ref{proofDecoupE1}) and (\ref{proofDecoupE2}) leads to
\begin{equation*}
\int_{\mathbb{R}^{2N}} \LR{\absolutevalue{D_j\Psi^{\mathrm{2D}}_{k}}^2 +V_{\varepsilon}(\mathbf{x}_j)\absolutevalue{\Psi_{k}^{\mathrm{2D}}}^2 } = e_{\varepsilon} + \int_{\mathbb{R}^{2N}}\absolutevalue{D_j\Phi_{k}}^2 U^2_{\varepsilon} + \int_{\mathbb{R}^{2N}} |x_j|^2 \absolutevalue{\Psi^{\mathrm{2D}}_{k}}^2,
\end{equation*}
and we complete the proof by summing over $j$.
\end{proof}

Returning to (\ref{+phase}), we have 
\begin{equation*}\label{diffgrad}
    \LR{D_j\Phi_{k}} e^{\mathrm{i}\alpha S} =  -\mathrm{i} \grad_{\mathbf{x}_j}\LR{\Phi_k e^{\mathrm{i}\alpha S }}  \quad \text{ in }  \quad \mathcal{D}'(\Lambda_0) .
\end{equation*}
Then Proposition \ref{decoupE} becomes
\begin{equation}\label{proofLower1}
    \lambda_k^{2\mathrm{D}} = N e_{\varepsilon} + \sum\limits_{j=1}^N \int_{\Lambda_0} \absolutevalue{ \grad_{\mathbf{x}_j}{\phi_{k}} }^2 U^2_{\varepsilon} + \sum\limits_{j=1}^N \int_{\mathbb{R}^{2N}} |x_j|^2 \absolutevalue{\Psi^{\mathrm{2D}}_{k}}^2,
\end{equation}
where 
$$\phi_{k}(\mathbf{x}_1,\dots,\mathbf{x}_N) \colonequals \Phi_k(\mathbf{x}_1,\dots,\mathbf{x}_N) e^{\mathrm{i}\alpha S(\mathbf{x}_1,\dots,\mathbf{x}_N) }. $$
We next rescale the tightly-confined space variables and denote 
\begin{equation}\label{eq:rescale func}
\phi_{\varepsilon,k}(x_1,y_1,\dots,x_N,y_N) \colonequals \phi_k(x_1,\sqrt{\varepsilon} y_1,\dots,x_N,\sqrt{\varepsilon}y_N). 
\end{equation}
Then Equation (\ref{proofLower1}) becomes
\begin{equation}\label{proofLower4}
    \lambda_k^{2\mathrm{D}} 
    = 
     N e_{\varepsilon} 
    + \sum\limits_{j=1}^N  \int_{\Lambda_0} \LR{\absolutevalue{\partial_{x_j}\phi_{\varepsilon,k}}^2 + \frac{1}{\varepsilon} \absolutevalue{\partial_{y_j}\phi_{\varepsilon,k}}^2}U_{1}^2 + \sum\limits_{j=1}^N \int_{\mathbb{R}^{2N}} |x_j|^2 \absolutevalue{\phi_{\varepsilon,k}}^2U^2_1,
\end{equation}
where $U_1$ is defined by setting $\varepsilon = 1$ in~\eqref{defU}.

Applying the upper bounds in Theorem \ref{thmUpperB} to Equation (\ref{proofLower4}), yields
\begin{equation}\label{proofLower5}
    \sum\limits_{j=1}^N  \int_{\Lambda_0} \LR{\absolutevalue{\partial_{x_j}\phi_{\varepsilon,k}}^2 + \frac{1}{\varepsilon} \absolutevalue{\partial_{y_j}\phi_{\varepsilon,k}}^2}U_{1}^2 + \sum\limits_{j=1}^N \int_{\mathbb{R}^{2N}} |x_j| ^2 \absolutevalue{\phi_{\varepsilon,k}}^2U^2_1 \le \lambda^{1\mathrm{D}}_k,
\end{equation}
where the right-hand side does not depend on $\varepsilon$.

Intuitively, it follows from the above that $\phi_{\varepsilon,k}$ will have less and less dependence on the $y$-coordinates as $\varepsilon$ goes to $0$. 
In the limit, $\phi_{\varepsilon,k}$ will depend only on the $x$-coordinates and the left-hand side of \eqref{proofLower5} will reduce to the energy of $\phi_{\varepsilon,k}$ under the 1D Hamiltonian $H^{1\mathrm{D}} $. 
Asymptotically the left-hand side of~\eqref{proofLower5} will be larger than the corresponding eigenvalue $\lambda_k^{1\mathrm{D}}$ (concluding the proof) if we can prove that $\phi_{\varepsilon,k}$ converges to a function of the $x$-coordinates only, vanishing on the diagonals.

To make this idea more rigorous, we express some properties of $\{{\phi_{\varepsilon,k}}\}_{\varepsilon} $ in terms of the weighted Sobolev spaces 
\begin{align*}
L^2_U\LR{\mathbb{R}^{2N}}&\colonequals L^2\LR{\mathbb{R}^{2N} ; U_1^2(y_1,\dots,y_N)\d \mathbf{x}_1\cdots\d \mathbf{x}_N},\nonumber\\
H_U^1\LR{\Lambda_{0}}&\colonequals H^1\LR{\Lambda_{0};U_1^2(y_1,\dots,y_N)\d \mathbf{x}_1\cdots\d \mathbf{x}_N} .
\end{align*}

\begin{prop}[\textbf{Convergence of $\{{\phi_{\varepsilon,k}}\}_{\varepsilon}$}] \label{ConvPhi}\mbox{}\\
After possibly extracting a subsequence, 
$$ \phi_{\varepsilon,k} \to \psi_{0,k}$$
strongly in $L^2_U(\Omega^N)$ for any bounded subset $\Omega\subset \mathbb{R}^2$ and weakly in $H_U^1\LR{\Lambda_{0}}$ with $\Lambda_0$ as in~\eqref{lam0}.
Furthermore, the limit 
$$
\psi_{0,k}(\bx_1,\dots,\bx_N) = \psi_{0,k}(x_1,\ldots,x_N) \in H^1({\mathbb{R}^N\backslash \mathbb{D}^{1\mathrm{D}}})
$$ 
has no dependence on the $y$-coordinates, and it satisfies the following property:
\begin{equation}\label{proofLower7}
    \liminf_{\varepsilon \to 0} \LR{\lambda^{2\mathrm{D}}_k - Ne_{\varepsilon} } 
    \ge \sum\limits_{j=1}^N \LR{ \int_{\mathbb{R}^N\backslash \mathbb{D}^{1\mathrm{D}}} \absolutevalue{\partial_{{x}_j}{\psi_{0,k}}}^2 + \int_{\mathbb{R}^N} |x_j| ^2 \absolutevalue{\psi_{0,k}}^2 } .
\end{equation}
\end{prop}

\begin{proof}
From~\eqref{proofLower5} we deduce that $\{{\phi_{\varepsilon,k}}\}_{\varepsilon} $ is bounded in $ H_U^1({\Lambda_{0}})$. According to the Banach-Alaoglu Theorem, we can extract a (not relabeled) subsequence converging weakly to some function $\psi_{0,k} $ in the Hilbert space $H_U^1\LR{\Lambda_{0}}$. 
Because of the coefficient $\varepsilon^{-1}$ multiplying the $y$-derivatives in (\ref{proofLower5}), we deduce that the limit $\psi_{0,k}$ only depends on the $x$-coordinates.
Then we have
\begin{equation*}
     \liminf_{\varepsilon \to 0} \norm{\phi_{\varepsilon,k}}_{L_U^2\LR{\Lambda_{0}}}
    \ge \norm{\psi_{0,k}}_{L_U^2\LR{\Lambda_{0}}} 
    = \norm{\psi_{0,k}}_{L^2{({\mathbb{R}^N\backslash \mathbb{D}^{1\mathrm{D}}})}}
    = \norm{\psi_{0,k}}_{L^2{({\mathbb{R}^N})}},
\end{equation*}
\begin{equation} \label{dercontrol}
     \liminf_{\varepsilon \to 0} \norm{\partial_{x_j}\phi_{\varepsilon,k}}_{L_U^2\LR{\Lambda_{0}}}  
    \ge \norm{\partial_{x_j}\psi_{0,k}}_{L_U^2\LR{\Lambda_{0}}} 
    = \norm{\partial_{x_j}\psi_{0,k}}_{L^2{({\mathbb{R}^N\backslash \mathbb{D}^{1\mathrm{D}}})}},
    \quad \forall j. 
\end{equation}
Using Sobolev embeddings, $(\phi_{\varepsilon,k})_{\varepsilon}$ converges strongly to $\psi_{0,k} $ 
in the weighted Sobolev space 
$$
L^2_U\LR{ V } \colonequals  L^2\LR{ V \,;U_1^2(y_1,\dots,y_N)\d \mathbf{x}_1\cdots\d \mathbf{x}_N}
$$ 
for any bounded set $ V  \subset  \Lambda_0$, equivalently for any bounded set $ V  \subset  \mathbb{R}^{2N}$. 
After passing to a further subsequence, $(\phi_{\varepsilon,k})_{\varepsilon}$ also converges to $\psi_{0,k}$ almost everywhere on $ V $. By a diagonal argument, we can choose a subsequence such that $(\phi_{\varepsilon,k})_{\varepsilon}$ converges to $\psi_{0,k}$ almost everywhere on $\mathbb{R}^{2N}$.
Then, with the help of Fatou's lemma, we have 
\begin{equation} \label{potentialcontrol}
    \liminf_{\varepsilon \to 0} \int_{\mathbb{R}^{2N}} |x_j| ^2 \absolutevalue{\phi_{\varepsilon,k}}^2 U^2_{1}
    \ge \int_{\mathbb{R}^{2N}} |x_j| ^2 \absolutevalue{{\psi_{0,k}}}^2 U^2_{1} 
    = \int_{\mathbb{R}^N} |x_j| ^2 {\absolutevalue{\psi_{0,k}}}^2
    ,\quad \forall j.
\end{equation}
Combining (\ref{proofLower4}), (\ref{dercontrol}) and (\ref{potentialcontrol}), we obtain (\ref{proofLower7}) and conclude the proof.
\end{proof}

Applying the diamagnetic inequality to $\Phi_k$, we obtain 
$$
\abs{\phi_k} = \abs{\Phi_k} \in H^{1}_{loc}(\mathbb{R}^{2N})
$$  
and 
$$ \abs{D_j \Phi_k} \geq \abs{\grad_{\mathbf{x}_j}\abs{\Phi_k}}  = \abs{\grad_{\mathbf{x}_j}\abs{\phi_k}}.
$$
Then Proposition \ref{decoupE} implies
\begin{equation}\label{rmk}
     \lambda_k^{2\mathrm{D}} - N e_{\varepsilon} \ge \sum\limits_{j=1}^N \int_{\mathbb{R}^{2N}} \absolutevalue{\grad_{\mathbf{x}_j} \absolutevalue{\phi_{k}} }^2 U^2_{\varepsilon} + \sum\limits_{j=1}^N \int_{\mathbb{R}^{2N}} |x_j| ^2 \absolutevalue{\phi_{k}}^2U^2_{\varepsilon}.
\end{equation}
Using the upper bounds in Theorem \ref{thmUpperB} and rescaling the functions in $y$-direction, Inequality (\ref{rmk}) implies that $ \{ |{\phi_{\varepsilon,k}}| \}_{\varepsilon}$ is bounded in the weighted Sobolev space $H_U^1(\mathbb{R}^{2N}) $ instead of $H_U^1(\Lambda_0) $. 
Then a similar proof as that of Proposition \ref{ConvPhi} gives

\begin{prop}[\textbf{Convergence of $ \{ |{\phi_{\varepsilon,k}}| \}_{\varepsilon}$}] \label{propRmk}\mbox{}\\
After extracting a subsequence, $ \{|{\phi_{\varepsilon,k}}|\}_{\varepsilon} $ converges weakly in $ H^1_U(\mathbb{R}^{2N}) $ to $|{\psi_{0,k}}|$, which is in $ H^1(\mathbb{R}^{N})$.
\end{prop}

Note that the limit $\psi_{0,k}$ in Proposition \ref{ConvPhi} may not be unique, i.e. it depends on the extracted subsequence. We will later recover some form of uniqueness by proving that these limits are all eigenfunctions of $H^{1\mathrm{D}}$ corresponding to the eigenvalues $\lambda^{1\mathrm{D}}_k$. Hence, the limit $\psi_{0,k}$ is unique if and only if the energy level  $\lambda^{1\mathrm{D}}_k$ is non-degenerate.

\bigskip

In order to finish the proof for the target lower bounds, it principally remains to prove that the right-hand side of Inequation (\ref{proofLower7}) is bounded from below by $\lambda^{1\mathrm{D}}_k $, i.e. 
\begin{equation}
    \sum\limits_{j=1}^N \LR{ \int_{\mathbb{R}^N\backslash \mathbb{D}^{1\mathrm{D}}} \absolutevalue{\partial_{{x}_j}{\psi_{0,k}}}^2 + \int_{\mathbb{R}^N} |x_j| ^2 \absolutevalue{\psi_{0,k}}^2 }\ge \lambda^{1\mathrm{D}}_k . \label{simpLower}
\end{equation}
The inequality of (\ref{simpLower}) follows from the next series of lemmas. 
The first four aim at proving that the limit $\psi_{0,k}$ is in the quadratic form domain of $H^{1\mathrm{D}}$, i.e. $\psi_{0,k} \in \mathcal{D}_q(H^{1\mathrm{D}})$. 
The main issue is to prove that the limit $\psi_{0,k}$ vanishes on diagonals (Proposition~\ref{vanish}). The last two lemmas are to check that the limit $\psi_{0,k}$ is indeed an eigenfunction of $ H^{1\mathrm{D}} $ corresponding to the eigenvalue $\lambda^{1\mathrm{D}}_k$.

\begin{lemma}[\textbf{Being $H^1$ accross diagonals}] \label{inH1}\mbox{}\\
Consider $\varphi \in L^2_{\mathrm{sym}}(\mathbb{R}^N)$ such that $ \varphi \in H^1({\mathbb{R}^N\backslash \mathbb{D}^{1\mathrm{D}}})$. Then $\varphi$ is in $H^1(\mathbb{R}^N)$.
\end{lemma}

\begin{proof}
We have that $\varphi$ is in $H^1$ outside of an interface, and continuous across the interface, because of the symmetry constraint. That $\varphi$ is globally in $H^1$ is then a classical result.

It suffices to show that~\cite[Theorem 3 in Section 5.8.2]{Evans-98}, for any unit $\mathbf{e}\in \mathbb{R}^N$ and $h\in (0,1)$, there exists some constant $C$ independent of $h$ and $\mathbf{e}$ such that
\begin{equation*}
    \int_{\mathbb{R}^N} \absolutevalue{ \frac{\varphi(\mathbf{X}+h\mathbf{e}) - \varphi(\mathbf{X})}{h}}^2 \d \mathbf{X}\le C.
\end{equation*}
Since $\varphi$ is symmetric on $\mathbb{R}^N$, it is equivalent to prove that, under the same notations,
\begin{equation*}
    \int_{\Lambda} \absolutevalue{ \frac{\varphi(\mathbf{X}+h\mathbf{e}) - \varphi(\mathbf{X})}{h}}^2 \d \mathbf{X}\le C,
\end{equation*}
where 
$$
\Lambda \colonequals \{ ({x}_1,\dots,{x}_N) \in \mathbb{R}^{N} \mid x_1 < \cdots <x_N \}. 
$$
Fix $h\in(0,1)$ and a unit $\mathbf{e}$. Let $$\mathbf{X} = (x_1,\dots,x_N) \in \Lambda \quad \mbox{ and } \quad \mathbf{X}+h\mathbf{e} = (x'_1,\dots,x'_N) .$$
There exists a permutation $\tau$ on $\{ 1,\dots,N \}$ such that 
$$\mathbf{X}'\colonequals (x'_{\tau(1)},\dots,x'_{\tau(N)}) \in \overline{\Lambda} .$$ 
We define 
\begin{equation*}
    h(\mathbf{X}) \colonequals \abs{\mathbf{X}'- \mathbf{X}}, \quad \mathbf{e}(\mathbf{X}) \colonequals \frac{\mathbf{X}'- \mathbf{X}}{h(\mathbf{X})}.
\end{equation*}
Then we have 
$$ \varphi \LR{\mathbf{X}+h(\mathbf{X})\mathbf{e}(\mathbf{X})} = \varphi (\mathbf{X}') = \varphi (\mathbf{X}+h\mathbf{e}). $$
Under this construction, $\mathbf{X}+h(\mathbf{X})\mathbf{e}(\mathbf{X})$ is actually in $\Lambda$ for almost every $\mathbf{X}\in \Lambda$, that is, except for a zero measure set in $\mathbb{R}^N$, $\mathbf{X}+h(\mathbf{X})\mathbf{e}(\mathbf{X})$ and $\mathbf{X}$ are both in $\Lambda$. 
Since a zero-measure set does not matter in our proof, we will ignore it from now on.

On the other hand, the domain $\Lambda$ is clearly convex, and  $\varphi$ is in $H^1(\Lambda)$. So we can apply the Newton-Leibniz theorem \cite[Section 4.9.2]{evans2015measure} to $\varphi$ on $\Lambda$ as follows:
\begin{equation}\label{NLT}
    \varphi(\mathbf{X}+h(\mathbf{X})\mathbf{e}(\mathbf{X})) - \varphi(\mathbf{X}) = h(\mathbf{X}) \int_0^1 \grad \varphi (\mathbf{X} + t \cdot h(\mathbf{X})\mathbf{e}(\mathbf{X})) \cdot \mathbf{e}(\mathbf{X}) \d t.
\end{equation}
One can check that  $|{h(\mathbf{X})}| \le h$. Thus, the identity (\ref{NLT}) above leads to
\begin{equation}\label{afterNLT}
    \begin{split}
        \absolutevalue{\frac{\varphi(\mathbf{X}+h\mathbf{e}) - \varphi(\mathbf{X}) }{h}}^2 & = \absolutevalue{\frac{\varphi(\mathbf{X}+h(\mathbf{X})\mathbf{e}(\mathbf{X})) - \varphi(\mathbf{X}) }{h}}^2 \\
        & \le \absolutevalue{ \int_0^1 \grad \varphi (\mathbf{X} + t \cdot h(\mathbf{X})\mathbf{e}(\mathbf{X})) \cdot \mathbf{e}(\mathbf{X}) \d t}^2\\
        & \le \int_0^1 \absolutevalue{\grad \varphi}^2(\mathbf{X} + t \cdot h(\mathbf{X})\mathbf{e}(\mathbf{X})) \d t.
    \end{split}
\end{equation}
It is not hard to verify that for each $\mathbf{e}$ and $h$, the function 
$$f_{\mathbf{e},h}^t: \mathbf{X} \mapsto \mathbf{X} + t \cdot h(\mathbf{X})\mathbf{e}(\mathbf{X})$$ 
is injective or trivial in the domain 
$$ \Lambda_{\sigma} \colonequals \{ \mathbf{X} \in \Lambda : \mathbf{X}+h\mathbf{e} \in \sigma(\Lambda_{}) \} \subset \Lambda $$ 
for each permutation $\sigma$ on $\{1,\dots, N\}$, where 
$$\sigma(\Lambda_{}) \colonequals \{ ({x}_1,\dots,{x}_N) \in \mathbb{R}^{N} \mid (x_{\sigma(1)} , \dots ,x_{\sigma(N)}) \in \Lambda \}.$$
With this observation, we can integrate over $\Lambda$ on both sides of~(\ref{afterNLT}) and obtain  
\begin{equation*}
    \begin{split}
        \int_{\Lambda} \absolutevalue{ \frac{\varphi(\mathbf{X}+h\mathbf{e}) - \varphi(\mathbf{X})}{h}}^2 \d \mathbf{X} & \le \int_{\Lambda} \int_0^1 \absolutevalue{\grad \varphi}^2(\mathbf{X} + t \cdot h(\mathbf{X})\mathbf{e}(\mathbf{X})) \d t \d \mathbf{X}\\
        & = \sum_{\sigma} \int_0^1 \int_{\Lambda_{\sigma}} \absolutevalue{\grad \varphi}^2(\mathbf{X} + t \cdot h(\mathbf{X})\mathbf{e}(\mathbf{X}))\d \mathbf{X} \d t \\
        & \le N! \cdot  \int_0^1 \int_{\Lambda} \absolutevalue{\grad \varphi}^2(\mathbf{X})\d \mathbf{X} \d t \\
        & \le N! \cdot  \norm{\varphi}_{H^1(\Lambda)}^2,
    \end{split}
\end{equation*}
which concludes the proof of Proposition~\ref{inH1}.
\end{proof}

From Proposition \ref{ConvPhi}, we know that the limit $\psi_{0,k} \in H^1({\mathbb{R}^N\backslash \mathbb{D}^{1\mathrm{D}}})$, and that with the help of the upper bounds in Theorem \ref{thmUpperB}, $x_j\psi_{0,k} \in L^2(\mathbb{R}^N)$ for all $j$.  
Since $\Psi_k^{2\mathrm{D}}$, $U_{\varepsilon}$ and $S$ are all symmetric, 
$$
\phi_{k} = \Psi_k^{2\mathrm{D}} U_{\varepsilon}^{-1} e^{\mathrm{i}\alpha S}
$$
is clearly symmetric, and so are $\phi_{\varepsilon,k}$ and its limit $\psi_{0,k}$. 
With the help of Proposition~\ref{inH1}, we obtain that the limit $\psi_{0,k} \in H^1(\mathbb{R}^{N})$.
Therefore, to prove that $\psi_{0,k} \in \mathcal{D}_q(H^{1\mathrm{D}}) $, it remains to prove that $ \psi_{0,k}  $ has zero trace on $\mathbb{D}^{1\mathrm{D}} $, i.e. $\psi_{0,k} \in H^1_0({\mathbb{R}^N\backslash \mathbb{D}^{1\mathrm{D}}}) $.
Since we already know that $\psi_{0,k}$ is symmetric, it suffices to prove that the limit $ \psi_{0,k}$ vanishes on the diagonal $\{x_1=x_2\}$.

\begin{lemma}[\textbf{Estimate near $\{ x_1=x_2 \}$}]\mbox{} \label{nearDiag}\\
Consider $U_1$ defined in~\eqref{defU} for $\varepsilon=1$ and $\phi_{\varepsilon,k} $ defined in~\eqref{eq:rescale func}. Then we have
    \begin{equation*}\label{afterseparate}
   \int_{G^x_{\varepsilon} \times G^y} { \absolutevalue{\phi_{\varepsilon,k}}^2 U_{1}^2 } \le C' \varepsilon
\end{equation*}
for some constant $C'>0$  independent of $\varepsilon$, where
\begin{equation*}\label{defGx}
    G^x_{\varepsilon} \colonequals \left\{ (x_1,\dots,x_N)\in\mathbb{R}^N : \left| x_1-x_2\right| < {\sqrt{\varepsilon}}/{2} , \, \left| x_1+x_2-2x_j\right| > \sqrt{\varepsilon},\; \forall j \ge 3 \right\}
\end{equation*}
and
$$ G^y \colonequals \left\{ (y_1,\dots,y_N)\in\mathbb{R}^N : |y_j| < {1/4} , \forall j \right\}. $$ 
\end{lemma}
\begin{proof}
Thanks to the upper bounds in Theorem \ref{thmUpperB} and the energy decoupling in Proposition \ref{decoupE}, there is a constant $C$ independent of $\varepsilon$ such that 
\begin{equation}\label{propProof1}
\sum_{j=1}^N\int_{\mathbb{R}^{2N}}{ \absolutevalue{D_j\Phi_{k}}^2 } U^2_{\varepsilon} \le C,
\end{equation}
where $\Psi_k^{2\mathrm{D}} = \Phi_k U_{\varepsilon} $ with the same notation as above.

Consider the region where $y$ is small, i.e. for $\beta >0$, consider 
$$
\Omega_{\beta} \colonequals \left\{(x,y)\in\mathbb{R}^2 \mid \absolutevalue{y} < \beta  \right\}, 
$$ 
and observe that in the region $\Omega_{\beta}^N$ (recall the definition of $u_{\varepsilon} $ in~\eqref{defu}), 
\begin{equation} \label{Ucontrol}
    u^{2N}_{\varepsilon}(\beta)  \le U_{\varepsilon}^2 \le \LR{\sqrt{\pi \varepsilon}}^{-N},
\end{equation} 
then we can take out the $y$-direction ground state $U_{\varepsilon}$ from the integral in (\ref{propProof1}) as follows:
\begin{equation}\label{afterUout}
    \LR{\sqrt{\pi \varepsilon}}^{-N}\sum_{j=1}^N\int_{\Omega_{\beta}^N}{ \absolutevalue{D_j\Phi_{k}}^2 } \le C    e^{\frac{\beta^2N}{\varepsilon}}.
\end{equation}

We define an open set $E$ in $ \mathbb{R}^{2(N-1)}$,
$$
E^{} \colonequals \left\{ (\mathbf{y}_2, \dots,\mathbf{y}_{N}) \in \mathbb{R}^{2(N-1)} : \mathbf{y}_j = (x_j,y_j), \, |y_j|<\beta \right\},
$$
and then define 
\begin{multline*}
    E_\beta \colonequals \Big\{ \left(\mathbf{x}_1,\dots,\mathbf{x}_N\right) \in \mathbb{R}^{2N}: \left(\mathbf{x}_1+\mathbf{x}_2, \mathbf{x}_3,\dots,\mathbf{x}_N\right) \in E, \\
    \absolutevalue{\mathbf{x}_1-\mathbf{x}_2} < \min \left\{ \beta, \vert\mathbf{x}_1+\mathbf{x}_2 - 2 \mathbf{x}_3\vert, \dots, \vert\mathbf{x}_1+\mathbf{x}_2 - 2 \mathbf{x}_N\vert \right\}  \Big\}.
\end{multline*}
One can check that $E^{}_{\beta} $ is contained in $\Omega_{\beta}^N$. Applying the isolated Hardy inequality in Theorem~\ref{isolatedHineq} to the function $ \Phi_k$ and the domain $E_{\beta}$, we obtain that
\begin{equation} \label{afterHardy}
    C_{\alpha}^{-1} \int_{E_{\beta}} \LR{ \absolutevalue{D_1 \Phi_k}^2 +\absolutevalue{D_2 \Phi_k}^2 } \ge \int_{E_{\beta}} \frac{\absolutevalue{\Phi_k}^2}{|{\mathbf{x}_1-\mathbf{x}_2}|^2} \ge \frac{1}{\beta^2} \int_{E_{\beta}} {\absolutevalue{\phi_k}^2}.
\end{equation}
Combining (\ref{Ucontrol}), (\ref{afterUout}) and (\ref{afterHardy}), we deduce
$$\int_{E_{\beta}} { \absolutevalue{\phi_k}^2 U_{\varepsilon}^2 } \le C C_{\alpha}^{-1} \beta^2   e^{\frac{\beta^2N}{\varepsilon}},$$
which after rescaling in the $y$-direction, $y\mapsto\sqrt{\varepsilon}y$, yields
\begin{equation} \label{afterRescale}
   \int_{\widetilde{E_{\beta}}} { \absolutevalue{\phi_{\varepsilon,k}}^2 U_{1}^2 } \le C C_{\alpha}^{-1} \beta^2   e^{\frac{\beta^2N}{\varepsilon}}
\end{equation}
for 
$$ \widetilde{E_{\beta}} \colonequals \left\{ \left(x_1,y_1,\dots,x_N,y_N\right) \in \mathbb{R}^{2N} :  \left(x_1,\sqrt{\varepsilon}y_1,\dots,x_N,\sqrt{\varepsilon}y_N\right) \in E_{\beta} \right\} .$$
Letting 
$$ \beta = \sqrt{\varepsilon}, $$
~\eqref{afterRescale} becomes 
\begin{equation*} \label{}
   \int_{\widetilde{E_{\beta}}} { \absolutevalue{\phi_{\varepsilon,k}}^2 U_{1}^2 } \le C'\varepsilon  ,
\end{equation*}
for some constant $C'$ independent of $\varepsilon$. One can check that $G^x_{\varepsilon}\times G^y $ is contained in $\widetilde{E_{\beta}}$ when $\beta=\sqrt{\varepsilon}$, which concludes the proof of Lemma~\ref{nearDiag}.
\end{proof}

We shall use a simple consequence of the Sobolev embedding:

\begin{lemma}[\textbf{Bounds in $W^{1,p} (\R^N)$}] \label{W1pbounded}\mbox{}\\
    Let $\{\psi_{\varepsilon}\}_{\varepsilon}$ be bounded in $H^1(\mathbb{R}^N)$. Then $ \{ |\psi_{\varepsilon}|^2 \}_{\varepsilon} $ is bounded in $W^{1,p}(G) $ for any bounded open subset $G$ in $\mathbb{R}^N $ and for $p \in [1 , N/(N-1)]$.
\end{lemma}
\begin{proof}
    Using Sobolev embeddings, 
    $$ H^1(G) \subset L^q(G) $$
    for any bounded open subset $G$ in $\R^{N}$ and for $q\in[1, 2N/(N-2)]$, and the embeddings are continuous. 
    Hence, $\{ |\psi_{\varepsilon}|^2 \}_{\varepsilon}$ is bounded in $L^p(G)$ for $p\in[1, N/(N-2)]$. 
    Notice that
    $$ \grad \LR{ \left| \psi_{\varepsilon}\right|^2}  = \psi_{\varepsilon} \grad \overline{\psi_{\varepsilon}} + \overline{\psi_{\varepsilon}} \grad \psi_{\varepsilon} . $$
    Then, to prove $\{ \grad (|\psi_{\varepsilon}|^2 )\}_{\varepsilon}$ is bounded in $ L^p(G)$, it suffices to prove $\{ \psi_{\varepsilon}\grad \psi_{\varepsilon}\}_{\varepsilon}$ is bounded in $ L^p(G)$. Using H\"older's inequality, yields
    \begin{equation}\label{afterHolder}
        \int_{G} \left| \psi_{\varepsilon}\grad \psi_{\varepsilon} \right|^p 
        \le \LR{\int_{G} \left| \grad \psi_{\varepsilon} \right|^{p\cdot \frac{2}{p}} }^{\frac{p}{2}} \LR{\int_{G} \left| \psi_{\varepsilon} \right|^{p\cdot s } }^{\frac{1}{s}} 
        = \left\Vert \grad \psi_{\varepsilon} \right\Vert^p_{L^2(G)}  \left\Vert \psi_{\varepsilon} \right\Vert^p_{L^{ps}(G)}
    \end{equation}
    for $p \in [1,\min\{2, N/(N-2)\}]$ and for $s$ such that
    $$ \frac{1}{s} + \frac{p}{2} = 1. $$
    One can check that $ ps \in [1, 2N/(N-2)] $ when $p\in [ 1, N/(N-1)]$, which shows $\{ \psi_{\varepsilon} \}_{\varepsilon} $ is bounded in $L^{ps}(G)$ because of the Sobelev embeddings mentioned above. 
    Therefore, it follows from~\eqref{afterHolder} that $\{ \psi_{\varepsilon}\grad \psi_{\varepsilon}\}_{\varepsilon}$ is bounded in $ L^p(G)$, which concludes the proof.
\end{proof}

In the following, we would like to concentrate on the properties of $\phi_{\varepsilon,k}$ in the $x$-direction, so we consider the function $\psi_{\varepsilon,k}$ on $\mathbb{R}^N$ defined by

\begin{equation}\label{psi1d}
    \psi_{\varepsilon,k} (x_1,\dots,x_N) \colonequals  \int_{G^y}  \absolutevalue{\phi_{\varepsilon,k}} (\mathbf{x}_1,\dots, \mathbf{x}_N)  U_1^2 (y_1,\dots, y_N) \d y_1\cdots \d y_N  .
\end{equation}
Then we have
\begin{lemma}[\textbf{Estimate on $\{ x_1=x_2 \}$}]\label{step2}\mbox{}\\
    Let $K$ be a bounded open subset in $\mathbb{R}^{N-1}$. For a constant $\gamma \ge 0$, we define 
    \begin{equation*}\label{defKgamma}
        K_{\gamma} \colonequals \left\{ (x_2,x_3,\dots,x_N) \in K : | x_2-x_j| > \sqrt{\gamma}/2, \forall j \ge 3\right\}.    
    \end{equation*}
    Then when $\gamma\ge \varepsilon$, we have
    \begin{equation*}
     \int_{K_{\gamma}} \absolutevalue{\psi_{\varepsilon,k}}^2(x_2,x_2,x_3,\dots,x_N)\d x_2 \cdots \d x_N 
    \le C_K \varepsilon^{\frac{s}{2}}\label{newEstm}
\end{equation*}
for some constants $C_K>0$  and $s \in (0,1)$ both independent of $\varepsilon$.
\end{lemma}

\begin{proof}
Using the Cauchy-Schwarz inequality, we have
\begin{equation}
    \absolutevalue{\psi_{\varepsilon,k}  }^2 \le C_1 \int_{G^y} \absolutevalue{\phi_{\varepsilon,k}}^2 U_{1}^2   \d y_1\cdots \d y_N  ,\label{afterCS}
\end{equation}
where 
\begin{equation}\label{defC1}
    C_1 \colonequals \int_{G^y} U^2_1 \d y_1\cdots \d y_N
\end{equation}
is a constant independent of $\varepsilon$ and $G^y$ is defined in Lemma \ref{nearDiag}.
And this implies that $\{\psi_{\varepsilon,k}\}_{\varepsilon}$ is bounded in $L^2(\mathbb{R}^N)$:
\begin{equation*}
    \int_{\mathbb{R}^N} \absolutevalue{\psi_{\varepsilon,k}}^2 
    \le C_1 \int_{\mathbb{R}^N}  \int_{G^y} \absolutevalue{\phi_{\varepsilon,k}}^2 U^2_1  
    \le C_1. \label{newFcontrol0}
\end{equation*}
On the other hand, we can control the derivative of $\psi_{\varepsilon,k}$ as follows:
\begin{equation} \label{newFcontrolder}
\begin{split}
    \int_{\mathbb{R}^N} \absolutevalue{\partial_{x_j}\psi_{\varepsilon,k}}^2 & \le \int_{\mathbb{R}^N} \LR{\int_{G^y} \absolutevalue{\partial_{x_j}\absolutevalue{\phi_{\varepsilon,k}}} U_1^2 }^2
    \\
    & \le C_1  \int_{\mathbb{R}^N} {\int_{G^y} \absolutevalue{\partial_{x_j}\absolutevalue{\phi_{\varepsilon,k}}}^2 U_1^2 } \quad \text{(Cauchy-Schwarz inequality)}\\
    & \le C_1  \int_{\mathbb{R}^{2N}} { \absolutevalue{\partial_{x_j}\absolutevalue{\phi_{k}}}^2 U_{\varepsilon}^2 }. \quad \text{(rescaling $y\mapsto (\sqrt{\varepsilon})^{-1}y $)}
\end{split}
\end{equation}
With the help of ~\eqref{rmk}, the estimate in ~\eqref{newFcontrolder} becomes
\begin{equation}\label{propProof6}
    \int_{\mathbb{R}^N} \absolutevalue{\partial_{x_j}\psi_{\varepsilon,k}}^2 \le C_1( \lambda_k^{2\mathrm{D}} - Ne_{\varepsilon}).
\end{equation}
Combining the upper bounds in Theorem~\ref{thmUpperB} to the estimate in (\ref{propProof6}) above, we can then deduce that $\{ \psi_{\varepsilon,k} \}_{\varepsilon} $ is bounded in $H^{1}(\mathbb{R}^{N})$. 
With the help of Lemma \ref{W1pbounded}, we know that $\{|{\psi_{\varepsilon,k}}|^2\}_{\varepsilon} $ is bounded in $W^{1,p}(G) $ for any bounded subset $G$ in $\mathbb{R}^N $ and $p\in [1, N/(N-1)]$.

Consider the centre of mass and relative coordinates, 
$$
R=\frac{x_1+x_2}{2}\mbox{ and } r=x_1-x_2,
$$ 
and let 
$$ 
\psi_{\varepsilon,k}^{{CM}}(r, R,x_3,\dots,x_N) = \psi_{\varepsilon,k}(R+r/2, R-r/2,x_3,\dots,x_N) , 
$$ 
then $\{|{\psi_{\varepsilon,k}^{{CM}}}|^2 \}_{\varepsilon}$ is still bounded in $W^{1,p}(G ) $.  Thus, we can apply the Newton-Leibniz theorem
to $|{\psi^{CM}_{\varepsilon,k}}|^2$:
\begin{equation}
    \absolutevalue{\psi^{CM}_{\varepsilon,k}}^2(r,R,x_3,\dots,x_N) - \absolutevalue{\psi^{CM}_{\varepsilon,k}}^2(0,R,x_3,\dots,x_N)
    = \int_{0}^{r} \partial_{r} \absolutevalue{\psi^{CM}_{\varepsilon,k}}^2(x,R,x_3,\dots,x_N)\d x.\label{NLthm}
\end{equation}
Let $K_{\gamma}\subset \R^{N-1}$ be defined in Lemma \ref{step2}, then integrating over $(-\sqrt{\varepsilon}/2, \sqrt{\varepsilon}/2) \times K_{\gamma}$ on both sides of the identity (\ref{NLthm}), we have
\begin{equation*}
    \begin{split}
    & \absolutevalue{\int_{\left(-\sqrt{\varepsilon}/2, \sqrt{\varepsilon}/2\right) \times K_{\gamma}} \left({\absolutevalue{\psi^{CM}_{\varepsilon,k}}^2(r,R,x_3,\dots,x_N) - \absolutevalue{\psi^{CM}_{\varepsilon,k}}^2(0,R,x_3,\dots,x_N) } \right) \differential r\differential R\differential x_3 \ldots \differential x_N}\\
    & = \quad \absolutevalue{\int_{\left(-\sqrt{\varepsilon}/2, \sqrt{\varepsilon}/2\right) \times K_{\gamma}} \LR{ \int_{0}^{r} \partial_{r} \absolutevalue{\psi^{CM}_{\varepsilon,k}}^2(x,R,x_3,\dots,x_N)\d x } \differential r\differential R\differential x_3 \ldots \differential x_N}\\
    & \le \quad  \int_{\left(-\sqrt{\varepsilon}/2, \sqrt{\varepsilon}/2\right) \times K_{\gamma}} \LR{\int_{-\sqrt{\varepsilon}/2}^{\sqrt{\varepsilon}/2} \absolutevalue{\partial_{r} \absolutevalue{\psi^{CM}_{\varepsilon,k}}^2(x,R,x_3,\dots,x_N)}\d x } \differential r\differential R\differential x_3 \ldots \differential x_N\\
    & = \quad  \sqrt{\varepsilon}  \cdot \int_{K_{\gamma}} \LR{\int_{-\sqrt{\varepsilon}/2}^{\sqrt{\varepsilon}/2} \absolutevalue{\partial_{r} \absolutevalue{\psi^{CM}_{\varepsilon,k}}^2(x,R,x_3,\dots,x_N)}\d x } \differential R\differential x_3\ldots \differential x_N\\
    & = \quad  \sqrt{\varepsilon} \cdot \int_{\left(-\sqrt{\varepsilon}/2, \sqrt{\varepsilon}/2\right) \times K_{\gamma}} {\absolutevalue{\partial_{r} \absolutevalue{\psi^{CM}_{\varepsilon,k}}^2(x,R,x_3,\dots,x_N)}\d x } \differential R\differential x_3\ldots \differential x_N.
    \end{split}
\end{equation*}
Using H\"older's inequality next, it turns out that
\begin{equation}
\begin{split}
& \absolutevalue{\int_{\left(-\sqrt{\varepsilon}/2, \sqrt{\varepsilon}/2\right) \times K_{\gamma}} \left({\absolutevalue{\psi^{CM}_{\varepsilon,k}}^2(r,R,x_3,\dots,x_N) - \absolutevalue{\psi^{CM}_{\varepsilon,k}}^2(0,R,x_3,\dots,x_N) } \right)}\\ 
& \le \quad \sqrt{\varepsilon} \cdot  \LR{\sqrt{\varepsilon}\absolutevalue{K_{\gamma}}}^{s} \norm{\partial_r\absolutevalue{\psi^{CM}_{\varepsilon,k}}^2}_{L^{p}\left(\left(-\sqrt{\varepsilon}/2, \sqrt{\varepsilon}/2\right) \times K_{\gamma}\right)} \\
& \le \quad \sqrt{\varepsilon} \cdot  \LR{\sqrt{\varepsilon}\absolutevalue{K_{}}}^{s} \norm{\absolutevalue{\psi^{CM}_{\varepsilon,k}}^2}_{W^{1,p}\left(\left(-\sqrt{\varepsilon}/2, \sqrt{\varepsilon}/2\right) \times K_{\gamma}\right)}, 
\end{split}\label{afterholder2}
\end{equation}
where $s= 1-p^{-1} \in (0,1)$ and $\absolutevalue{K}$ is the Lebesgue measure of $K$ in $\mathbb{R}^{N-1}$. 
Since $\{|{\psi_{\varepsilon,k}^{{CM}}}|^2 \}_{\varepsilon}$ is bounded in $W^{1,p}(G ) $, the estimate in~\eqref{afterholder2} implies \begin{equation}\label{1pdifference}
    \absolutevalue{\int_{\left(-\sqrt{\varepsilon}/2, \sqrt{\varepsilon}/2\right) \times K_{\gamma}} \left({\absolutevalue{\psi^{CM}_{\varepsilon,k}}^2(r,R,x_3,\dots,x_N) - \absolutevalue{\psi^{CM}_{\varepsilon,k}}^2(0,R,x_3,\dots,x_N) } \right)} \le C_{{K}}'\cdot  {\varepsilon}^{\frac{1+s}{2}}
\end{equation}
for some constant $C_K'>0$ independent of $\varepsilon$.

On the other hand, if 
$$\left(x_1-x_2, (x_1+x_2)/2, x_3, \dots,x_N\right) \in \left(-\sqrt{\varepsilon}/2, \sqrt{\varepsilon}/2\right) \times K_{\gamma}$$
then we have 
$$ (x_1,\dots,x_N) \in G_{\varepsilon}^x  $$
for all $\gamma\ge \varepsilon$, where $G_{\varepsilon}^x$ is defined in Lemma \ref{nearDiag}. 
Then the estimate in Lemma \ref{nearDiag} and Inequality~(\ref{afterCS}) imply
\begin{equation*}
    \int_{\left(-\sqrt{\varepsilon}/2, \sqrt{\varepsilon}/2\right) \times K_{\gamma}} \absolutevalue{\psi^{CM}_{\varepsilon,k}}^2 \le \int_{G_{\varepsilon}^x} \left|\psi_{\varepsilon,k}   \right|^2  \le C'' \varepsilon, \label{newFcontrol}
\end{equation*}
for some constant $C''>0$ independent of $\varepsilon$, which together with the controls in (\ref{1pdifference}) shows that
\begin{equation*}
    \int_{K_{\gamma}} \absolutevalue{\psi_{\varepsilon,k}}^2(x_2,x_2,x_3,\dots,x_N)\d x_2 \cdots \d x_N 
    = \int_{K_{\gamma}} \absolutevalue{\psi^{CM}_{\varepsilon,k}}^2(0,R,x_3,\dots,x_N)\differential R\differential x_3 \cdots \differential x_N
    \le C_K \varepsilon^{\frac{s}{2}}
\end{equation*}
for some constant $C_K>0$ independent of $\varepsilon$, concluding the proof. 
\end{proof}

\begin{prop}[\textbf{Vanishing on diagonals}] \label{vanish}\mbox{}\\
The limit $ \psi_{0,k}$ defined in Proposition \ref{ConvPhi} vanishes on the diagonals $\mathbb{D}^{1\mathrm{D}}$. 
\end{prop}

\begin{proof}
Notice that $|{\psi_{0,k}}|$ can be rewritten as
\begin{equation*}
\begin{split}
    \absolutevalue{\psi_{0,k}} (x_1,\dots,x_N) 
    = \frac{1}{C_1}  \int_{G^y} \absolutevalue{\psi_{0,k}} (x_1,\dots,x_N) U^2_1(y_1,\dots,y_N) \d y_1 \cdots \d y_N,
\end{split}
\end{equation*}
where $G^y$ is defined in Lemma \ref{nearDiag} and $C_1$ is defined in~\eqref{defC1}. 
Let $\psi_{\varepsilon,k}$ be defined in~\eqref{psi1d}.
Then for each bounded subset $G\subset \mathbb{R}^N$, using the Cauchy-Schwarz inequality, we have
\begin{equation*}
    \begin{split}
        \norm{\psi_{\varepsilon,k} - C_1\left|{\psi_{0,k}}\right|}^2_{L^2(G)} 
        & \le \int_{G} \LR{\int_{G^y} \Big\vert{ \absolutevalue{{\phi}_{\varepsilon,k}} - \absolutevalue{\psi_{0,k}} }\Big\vert U^2_1}^2\\
        & \le C_1 \norm{\absolutevalue{{\phi}_{\varepsilon,k}} - \absolutevalue{\psi_{0,k}}}_{L^2_U(\Omega^N)}^2\xrightarrow[\text{Proposition \ref{ConvPhi}}]{} 0, 
    \end{split}
\end{equation*}
where $\Omega$ is a bounded subset in $\mathbb{R}^{2}$ such that $\Omega^N$ contains $G\times G^y$.
Hence, $$\psi_{\varepsilon,k}\to C_1\left|{\psi_{0,k}}\right|$$
strongly in $L^2(G)$. 
Meanwhile, we know that $\{ \psi_{\varepsilon,k} \}_{\varepsilon} $ is bounded in $H^{1}(\mathbb{R}^{N})$ from the proof of Lemma~\ref{step2}.
Then according to the Banach-Alaoglu Theorem, $\{\psi_{\varepsilon,k}\}_{\varepsilon} $ also converges weakly to $C_1|\psi_{0,k}|$ in $H^{1}(\mathbb{R}^N ) $ after passing to a subsequence.
Since the trace operator $T$ defined in (\ref{traceop}) is continuous, $${\psi^{}_{\varepsilon,k}(x_2,x_2,\dots,x_N)} \to C_1 |{\psi^{}_{0,k}}|(x_2,x_2,\dots,x_N)$$
weakly in $L^2(\mathbb{R}^{N-1})$. 
Therefore, with the estimate in Lemma \ref{step2}, for each $\gamma>0$ independent of $\varepsilon$, we have 
\begin{equation*}
\begin{split}
    0 &\le \int_{K_{\gamma}} \absolutevalue{\psi_{0,k}}^2(x_2,x_2,x_3,\dots,x_N)\d x_2\cdots \d x_N   \\
    & \le \frac{1}{C_1^2} \liminf_{\varepsilon \to 0} \int_{K_{\gamma}} \absolutevalue{\psi_{\varepsilon,k}}^2(x_2,x_2,x_3,\dots,x_N)\d x_2 \cdots \d x_N \le 0,
\end{split}
\end{equation*}
which implies that ${\psi_{0,k}}(x_2,x_2,x_3,\dots,x_N) $ equals to $0$ in $L^2({K_{\gamma}}) $ for any bounded open subset $K$ in $\R^{N-1}$ and for any constant $\gamma>0$ independent of $\varepsilon$. Letting finally $\gamma \to 0$  after $\varepsilon \to 0$ it follows that ${\psi_{0,k}}(x_2,x_2,x_3,\dots,x_N) $ equals to $0$ in $L^2(\R^{N-1}) $, i.e.  $\psi_{0,k}$ vanishes on the diagonal $\{x_1=x_2\} $, which concludes the proof of Proposition \ref{vanish}.
\end{proof}

With the help of Proposition \ref{vanish}, the limit $\psi_{0,k} $ is indeed in $\mathcal{D}_q(H^{1\mathrm{D}})$. 
Moreover, Inequality (\ref{proofLower7}) implies
\begin{equation}\label{refined}
    \liminf_{\varepsilon\to 0} \LR{\lambda^{2\mathrm{D}}_{k} - N  e_{\varepsilon}} \ge \mathcal{E}^{1\mathrm{D}} (\psi_{0,k} ).
\end{equation}
Combining with the upper bounds in Theorem \ref{thmUpperB}, we deduce
\begin{equation}\label{afterinDomain}
    \mathcal{E}^{1\mathrm{D}} (\psi_{0,k} )\le \lambda^{1\mathrm{D}}_k.
\end{equation}
We next need the following strong $L^2$-compactness:

\begin{lemma}[\textbf{Strong $L^2$-compactness}]\label{normal}\mbox{}\\
With the same notation as in Proposition \ref{ConvPhi}, we have
$$ \phi_{\varepsilon,k} \to \psi_{0,k} $$
strongly in $L^2_U(\R^{2N})$. Hence, the limit $\psi_{0,k}$ is normalized in $L^2(\mathbb{R}^N)$.
\end{lemma}

\begin{proof}
Applying the diamagnetic inequality to the 2D eigenfunction $\Psi^{2\mathrm{D}}_{k} $, we obtain 
\begin{equation}
    \sum\limits_{j=1}^N \int_{\mathbb{R}^{2N}} \absolutevalue{\grad_{\mathbf{x}_j}\absolutevalue{\Psi^{\mathrm{2D}}_{k} }}^2 
    + \sum\limits_{j=1}^N \int_{\mathbb{R}^{2N}} V_{\varepsilon}(\mathbf{x}_j) \absolutevalue{\Psi^{\mathrm{2D}}_{k}}^2 \le \lambda^{\mathrm{2D}}_{k}. \label{directDiam}
\end{equation}
We next decompose $\absolutevalue{{\Psi}^{\mathrm{2D}}_{k}}$ into its orthogonal projection on $U_{\varepsilon}$ and a remainder:
\begin{equation}\label{decompose2D}
    \absolutevalue{{\Psi}^{\mathrm{2D}}_{ k}} (\mathbf{x}_1,\dots,\mathbf{x}_N) = \psi^{abs}_{\varepsilon , k}({x}_1,\dots,{x}_N) U_{\varepsilon}({y}_1,\dots,{y}_N) + \omega^{abs}_{\varepsilon,k} (\mathbf{x}_1,\dots,\mathbf{x}_N),
\end{equation}
where \begin{equation*}
    \psi^{abs}_{\varepsilon,k}({x}_1,\dots,{x}_N) = \int_{\mathbb{R}^N}\absolutevalue{\Psi_{k}^{\mathrm{2D}}} (\mathbf{x}_1,\dots,\mathbf{x}_N)U_{\varepsilon}({y}_1,\dots,{y}_N)\differential y_1\cdots\differential y_N. 
\end{equation*}
By definition, we can deduce $U_{\varepsilon}$ and $\omega_{\varepsilon,k}^{abs}$ are orthogonal on $y$-direction:
\begin{equation*}
    \int_{\R^N} \omega^{abs}_{\varepsilon,k} (\mathbf{x}_1,\dots,\mathbf{x}_N) U_{\varepsilon}({y}_1,\dots,{y}_N)\differential y_1\cdots\differential y_N = 0.
\end{equation*}
Hence, we have
\begin{equation*}
    1 = \norm{{\Psi}^{\mathrm{2D}}_{ k}}_{L^2(\R^{2N})} = \norm{\psi^{abs}_{\varepsilon,k}}_{L^2(\R^{N})} + \norm{\omega^{abs}_{\varepsilon,k} }_{L^2(\R^{2N})}
\end{equation*}
and then (\ref{directDiam}) above yields 
\begin{equation}\label{estiomega}
\begin{split}
    \lambda_k^{2\mathrm{D}} & \ge \sum\limits_{j=1}^N \int_{\mathbb{R}^{2N}} \absolutevalue{\partial_{y_j}\absolutevalue{\Psi^{\mathrm{2D}}_{k} }}^2 
    + \sum\limits_{j=1}^N \int_{\mathbb{R}^{2N}} \frac{1}{\varepsilon^2}\left|y_j\right|^2\absolutevalue{\Psi^{\mathrm{2D}}_{k}}^2 \\
    & =  N e_{\varepsilon} \cdot
    \norm{\psi^{abs}_{\varepsilon , k} }_{{L^2(\R^N)}}^2 + \sum\limits_{j=1}^N \int_{\mathbb{R}^{2N}} \absolutevalue{\partial_{y_j}\absolutevalue{\omega^{abs}_{\varepsilon,k} }}^2 
    + \sum\limits_{j=1}^N \int_{\mathbb{R}^{2N}} \frac{1}{\varepsilon^2}\left|y_j\right|^2\absolutevalue{\omega^{abs}_{\varepsilon,k} }^2\\
    & \ge  N e_{\varepsilon} \cdot
    \norm{\psi^{abs}_{\varepsilon , k} }_{{L^2(\R^N)}}^2  + (N+2)e_{\varepsilon} \cdot \norm{\omega^{abs}_{\varepsilon,k} }_{{L^2(\R^N)}}^2 \\
    & = Ne_{\varepsilon} + 2 e_{\varepsilon} \norm{\omega^{abs}_{\varepsilon,k} }_{{L^2(\R^N)}}^2,
\end{split}
\end{equation}
where in the last inequality we also use the fact that the second eigenvalue of $-\partial_y^2 + \varepsilon^{-2}y^2$ is equal to $3e_{\varepsilon}$.
Combining the above estimate~\eqref{estiomega} with the upper bounds in Theorem \ref{thmUpperB}, we can deduce that $\{\omega^{abs}_{\varepsilon,k}\}_{\varepsilon}$ converges strongly to 0 in $L^2(\mathbb{R}^{2N})$   as $\varepsilon$ goes to $0$.
After rescaling the $y$-coordinates, $y\mapsto \sqrt{\varepsilon}y$, the identity~\eqref{decompose2D} becomes
\begin{equation*}
    \absolutevalue{\phi_{\varepsilon,k}}U_{1}  = \psi^{abs}_{\varepsilon,k} U_{1} + \widetilde{\omega}^{abs}_{\varepsilon,k} ,
\end{equation*}
where $\phi_{\varepsilon,k} $ is the same one as in~\eqref{eq:rescale func} and 
$$
\widetilde{\omega}_{\varepsilon,k}^{abs} \left(x_1,y_1,\dots,x_N,y_N\right) = \LR{\sqrt{\varepsilon}}^{\frac{N}{2}} {\omega}_{\varepsilon,k}^{abs} \left(x_1, \sqrt{\varepsilon}y_1,\dots, x_N,\sqrt{\varepsilon}y_N\right). 
$$
We then note that 
$$ \norm{\widetilde{\omega}_{\varepsilon,k}^{abs}}_{L^2(\mathbb{R}^{2N} )} = \norm{\omega^{abs}_{\varepsilon,k}}_{L^2(\mathbb{R}^{2N} )}\to 0.  $$
With the same notation as in Proposition \ref{ConvPhi}, it follows that
\begin{equation*}
    \begin{split}
        \norm{ { \absolutevalue{\psi_{0,k}} - \psi_{\varepsilon,k}^{abs}}  }_{L^2_U(\Omega^N )}
        & \le \norm{ { \absolutevalue{\psi_{0,k}} - \absolutevalue{{\phi}_{\varepsilon,k}}}}_{L^2_{U}(\Omega^N )} + \norm{ {\absolutevalue{{\phi}_{\varepsilon,k}} -  \psi_{\varepsilon,k}^{abs}}  }_{L^2_{U}(\Omega^N )}  \\
        & \le \norm{ { \absolutevalue{\psi_{0,k}} - \absolutevalue{{\phi}_{\varepsilon,k}}}}_{L^2_{U}(\Omega^N )} + \norm{ \widetilde{\omega}_{\varepsilon,k}^{abs}}_{L^2(\mathbb{R}^{2N} )} \to 0,
    \end{split}
\end{equation*}
that is to say, $$ \psi_{\varepsilon,k}^{abs} \to \absolutevalue{\psi_{0,k}}$$strongly in $L^2({G} )$ for any bounded set ${G}\subset\mathbb{R}^N$.
Notice from the definition of $\psi^{abs}_{\varepsilon,k}$ and the Cauchy-Schwarz inequality that
\begin{equation*}
    \sum_{j=1}^N \norm{\partial_{x_j}\psi^{abs}_{\varepsilon,k}}^2_{L^2(\mathbb{R}^N )} \le \sum_{j=1}^N \norm{\partial_{x_j}\absolutevalue{{\phi}_{\varepsilon,k}}}^2_{L^2_U(\mathbb{R}^{2N} )}.
\end{equation*}
Since $\{|{\phi}_{\varepsilon,k}|\}_{\varepsilon}$ is bounded in $H^1_U(\mathbb{R}^{2N} )$ by Proposition \ref{propRmk}, we obtain that $\{ \psi^{abs}_{\varepsilon,k} \}_{\varepsilon}$ is bounded in $H^1(\mathbb{R}^N )$. 
Again using the Banach-Alaoglu Theorem,  $\{\psi^{abs}_{\varepsilon,k}\}_{\varepsilon} $ also converges weakly to $|\psi_{0,k}|$ in $H^1(\mathbb{R}^N )$ after passing to a subsequence, i.e. $\psi^{abs}_{\varepsilon,k} -|\psi_{0,k}| \rightharpoonup 0$ in $H^1(\mathbb{R}^N )$. 
Referring to \cite[Lemma 14]{Lewin-compact}, for a positive sequence $\{R_\varepsilon\}_{\varepsilon}$ going to infinity, again after passing to a subsequence, $\{\psi^{abs}_{\varepsilon,k} -|\psi_{0,k}|\}_{\varepsilon}$ satisfies the following local estimate: 
\begin{equation}
    \int_{{\absolutevalue{x_j}\le R_{\varepsilon}, \,\forall j}} \absolutevalue{\psi^{abs}_{\varepsilon,k} -\absolutevalue{\psi_{0,k}}}^2 (x_1,\dots,x_N) \d x_1 \dots,\d x_N \to 0. \label{cptPart}
\end{equation}
On the other hand, the orthogonality of $U_{\varepsilon}$ and $\omega ^{abs}_{\varepsilon}$ in the $y$-direction leads to 
\begin{equation}\label{othoxj}
    \int_{\mathbb{R}^N} |x_j| ^2 \absolutevalue{\psi^{abs}_{\varepsilon,k}}^2(x_1,\dots,x_N) \d x_1 \cdots\d x_N 
    \le  \int_{\mathbb{R}^{2N}} |x_j| ^2 \absolutevalue{\Psi^{2\mathrm{D}}_{k}}^2(\mathbf{x}_1,\dots,\mathbf{x}_N) \d \mathbf{x}_1 \cdots\d \mathbf{x}_N.
\end{equation}
Thanks to the upper bounds in Theorem \ref{thmUpperB} and the energy decoupling in Proposition \ref{decoupE}, we have
\begin{equation}\label{controlonPsi}
    \int_{\mathbb{R}^{2N}} |x_j| ^2 \absolutevalue{\Psi^{2\mathrm{D}}_{k}}^2(\mathbf{x}_1,\dots,\mathbf{x}_N) \d \mathbf{x}_1 \cdots\d \mathbf{x}_N \le C
\end{equation}
for some constant $C$ independent of $\varepsilon$. 
Then inequalities~\eqref{othoxj} and~\eqref{controlonPsi} together yield
\begin{equation*}
    \int_{\mathbb{R}^N} |x_j| ^2 \absolutevalue{\psi^{abs}_{\varepsilon,k}}^2(x_1,\dots,x_N) \d x_1 \cdots\d x_N \le C,
\end{equation*}
which implies
\begin{equation*}
    \sum_{j=1}^N \int_{\mathbb{R}^N} |x_j| ^2 \absolutevalue{\psi^{abs}_{\varepsilon,k} - \absolutevalue{\psi_{0,k}}}^2(x_1,\dots,x_N) \d x_1 \cdots \d x_N 
    \le \widetilde{C}
\end{equation*}
for some constant $\widetilde{C}$ independent of $\varepsilon$. It follows that 
\begin{equation}
    \int_{{ \exists j,\, \absolutevalue{x_j} > R_\varepsilon}} \absolutevalue{\psi^{abs}_{\varepsilon,k} - \absolutevalue{\psi_{0,k}}}^2(x_1,\dots,x_N) \d x_1 \cdots \d x_N 
    \le \frac{\widetilde{C}}{R_\varepsilon^2} \to 0. \label{nonlocalPart}
\end{equation}
Combining the local estimate (\ref{cptPart}) and the non-local estimate (\ref{nonlocalPart}), we claim that $\{\psi^{abs}_{\varepsilon,k}\}_{\varepsilon} $ converges strongly to $\abs{\psi_{0,k}}$ in the whole space, i.e. in $L^2(\mathbb{R}^N ) $.
Therefore, we have
\begin{equation*}
    \norm{\psi_{0,k}}_{L^2(\mathbb{R}^N )}^2  = \lim_{\varepsilon\to 0} \norm{\psi^{abs}_{\varepsilon,k}}_{L^2(\mathbb{R}^N )}^2  = \lim_{\varepsilon\to 0} \LR{\norm{\Psi^{2\mathrm{D}}_{k}}_{L^2(\mathbb{R}^{2N} )}^2 - \norm{\omega^{abs}_{\varepsilon,k}}_{L^2(\mathbb{R}^{2N} )}^2} = 1,
\end{equation*}
which together with the weak convergence 
$$\phi_{\varepsilon,k} \overset{L^2_U(\Lambda_0)}{\rightharpoonup}\psi_{0,k} $$ 
that we proved in Proposition \ref{ConvPhi} implies the strong convergence claimed in Lemma \ref{normal}.
\end{proof}

\begin{lemma}[\textbf{Limiting eigenfunctions}]\label{be1DEF}\mbox{}\\
The limit $\psi_{0,k}$ defined in Proposition \ref{ConvPhi} is an eigenfunction of $H^{1\mathrm{D}}$ corresponding to the eigenvalue $\lambda^{1\mathrm{D}}_k$.
\end{lemma}

\begin{proof}
Let $\{\phi_{\varepsilon,k}\}_{\varepsilon}$ converge to $\psi_{0,k}$ in the sense of the one in Proposition \ref{ConvPhi}. 
Then after passing to a subsequence, $\{\phi_{\varepsilon,j}\}_{\varepsilon}$ converges to some $\psi_{0,j}$ in the same sense for all $j\le k$.
Proposition \ref{normal} tells us that $\{\phi_{\varepsilon,j}\}_{\varepsilon}$ converges strongly to $\psi_{0,j} $ in $ L^2_U(\mathbb{R}^{2N})$. 
Therefore, for $j\neq l \le k$, we have
\begin{equation}\label{proofprop3}
    \begin{split}
        \langle \psi^{}_{0,j} | \psi^{}_{0,l} \rangle_{L^2(\mathbb{R}^{N})}  & =  \langle \psi^{}_{0,j} U_{1} | \psi^{}_{0,l} U_{1} \rangle_{L^2(\mathbb{R}^{2N})}  \\  
        & =  \lim_{\varepsilon\to 0} \langle \phi^{}_{\varepsilon,j} U_{1} | \phi^{}_{\varepsilon,l} U_{1} \rangle_{L^2(\mathbb{R}^{2N})}  \\
        & =  \lim_{\varepsilon\to 0 } \langle \Psi^{2\mathrm{D}}_{j} | \Psi^{2\mathrm{D}}_{l} \rangle_{L^2(\mathbb{R}^{2N})}  = 0,
    \end{split}
\end{equation}
i.e. $\{ \psi_{0,j} \}_{j=1}^k $ is an orthonormal set in $L^2(\mathbb{R}^N)$. 

We now prove the lemma by induction. Firstly, for $m=1$, it is clear that $$\mathcal{E}^{1\mathrm{D}} (\psi_{0,1}) \ge \lambda^{1\mathrm{D}}_1,$$ and with the upper bounds in (\ref{afterinDomain}) we obtain that $\psi_{0,1}$ is the eigenfunction corresponding to $\lambda^{1\mathrm{D}}_1$. 
Suppose that for $m=k-1$, the set of limit $\{\psi_{0,j}\}_{j=1}^{k-1}$  is a set of eigenfunctions corresponding to the eigenvalues $\{ \lambda^{1\mathrm{D}}_j \}_{j=1}^{k-1}$ respectively. 
According to the identity (\ref{proofprop3}), we have $$\psi_{0,k} \perp_{L^2(\mathbb{R}^N)} \{\psi_{0,j}\}_{j=1}^{k-1} ,$$ which implies $$\mathcal{E}^{1\mathrm{D}} (\psi_{0,k})  \ge \lambda^{1\mathrm{D}}_{k}.$$
With the help of the upper bounds in (\ref{afterinDomain}), this gives 
$$
\mathcal{E}^{1\mathrm{D}} (\psi_{0,k})  = \lambda^{1\mathrm{D}}_{k}.
$$
Then we can deduce that $\psi_{0,k}$ is an eigenfunction corresponding to $\lambda^{1\mathrm{D}}_{k}$. 
\end{proof}

Thanks to Lemma \ref{be1DEF}, Inequality (\ref{refined}) gives 
$$
\liminf_{\varepsilon\to 0} \LR{\lambda^{2\mathrm{D}}_{k} - N  e_{\varepsilon}} \ge \lambda_{k}^{1\mathrm{D}}
$$ 
which concludes the proof of Theorem~\ref{thmLowerB}.

\section{Relation between eigenfunctions}

In this section, we conclude the proof for the relation between eigenfunctions stated in Theorem \ref{thmEF}.

Similarly to Lemma \ref{normal}, we also have the corresponding strong $H^1$-compactness:
\begin{lemma}[\textbf{Strong $H^1$-compactness}]\label{cptH1}\mbox{}\\
    With the same notation as in Proposition \ref{ConvPhi}, we have
$$ \phi_{\varepsilon,k} \to \psi_{0,k} $$
strongly in $H^1_U(\Lambda_0)$. 
\end{lemma}
\begin{proof}
Since we have proved the corresponding strong $L^2$-compactness in Lemma \ref{normal}, it remains to prove that 
$$ \grad \phi_{\varepsilon,k} \to \grad \psi_{0,k} $$
strongly in $(L^2_U(\Lambda_0) )^{2N}$.

Recalling the estimate~\eqref{proofLower5}, it follows that
$$ \partial_{y_j} \phi_{\varepsilon,k} \to 0 = \partial_{y_j}  \psi_{0,k} $$
strongly in $L^2_U(\Lambda_0) $ for all $j$, and that 
\begin{equation}\label{estiONder}
\begin{split}
        \lambda^{1\mathrm{D}}_k & \ge \limsup_{\varepsilon \to 0}\sum\limits_{j=1}^N  \LR{ \int_{\Lambda_0} {\absolutevalue{\partial_{x_j}\phi_{\varepsilon,k}}^2}U_{1}^2 + \int_{\mathbb{R}^{2N}} |x_j| ^2 \absolutevalue{\phi_{\varepsilon,k}}^2U^2_1}\\
        & \ge \liminf_{\varepsilon \to 0}  \sum\limits_{j=1}^N  \LR{\int_{\Lambda_0} {\absolutevalue{\partial_{x_j}\phi_{\varepsilon,k}}^2}U_{1}^2 + \int_{\mathbb{R}^{2N}} |x_j| ^2 \absolutevalue{\phi_{\varepsilon,k}}^2U^2_1}.
\end{split}
\end{equation}
According to the proof of Proposition \ref{ConvPhi} and Lemma \ref{be1DEF}, the right-hand side of~\eqref{estiONder} is bounded from below by $\lambda_k^{1\mathrm{D}}$, which implies that
\begin{equation*}
    \lim_{\varepsilon \to 0} \left\Vert \grad \phi_{\varepsilon,k} \right\Vert_{(L^2_U(\Lambda_0))^{2N}} = \left\Vert \grad \psi_{0,k}^{} \right\Vert_{(L^2_U(\Lambda_0))^{2N}}.
\end{equation*}
Combining it with the fact proved in Proposition \ref{ConvPhi} that $$\grad \phi_{\varepsilon,k} \to \grad \psi_{k}^{1\mathrm{D}} $$ weakly in $(L^2_U(\Lambda_0))^{2N}$, we obtain that such convergence is in fact strong, which concludes the proof of Lemma \ref{cptH1}.
\end{proof}

\begin{proof}[Proof of Theorem \ref{thmEF}]
Using a diagonal extraction, we can construct a sequence $\{\phi_{\varepsilon,k}\}_{\varepsilon}$ which converges to some $\psi_{0,k}$ in the sense of~Proposition \ref{ConvPhi} for all $k$. 
We know from Proposition \ref{be1DEF} that these limit $\psi_{0,k}$ are eigenfunctions corresponding to the 1D eigenvalues $\lambda_k^{1\mathrm{D}}$ respectively, and from now on we denote them by $\psi^{1\mathrm{D}}_k$.
Similarly to the proof of Proposition \ref{be1DEF} above, we can deduce that $\{ \psi^{1\mathrm{D}}_k\}_k $ is an orthonormal set in $L^2(\mathbb{R}^N)$, hence, it is an orthonormal basis.

Recall the notation in the proof of Proposition \ref{ConvPhi}, we have
\begin{equation*}    \begin{split}
        \left\Vert\Psi_k^{2\mathrm{D}} -  \psi^{1\mathrm{D}}_k  U_{\varepsilon} \right\Vert_{L^2(\mathbb{R}^{2N})}
         & = \norm{\phi_k U_{\varepsilon} e^{-\mathrm{i}\alpha S} - \psi^{1\mathrm{D}}_kU_{\varepsilon} }_{L^2(\mathbb{R}^{2N})}  \\
        & \le \norm{ \left({\phi_k  - \psi^{1\mathrm{D}}_k }\right)  U_{\varepsilon} e^{-\mathrm{i}\alpha S}  }_{L^2(\mathbb{R}^{2N})} + \norm{\psi^{1\mathrm{D}}_k  U_{\varepsilon} \left(e^{-\mathrm{i}\alpha S} -1\right)}_{L^2(\mathbb{R}^{2N})} \\
        & \le \norm{ \phi_{\varepsilon, k} - \psi^{1\mathrm{D}}_k}_{L^2_U(\mathbb{R}^{2N})} +  \norm{\psi^{1\mathrm{D}}_k\left( e^{-\mathrm{i}\alpha\sqrt{\varepsilon} S_{}} - 1 \right)}_{L^2_U(\mathbb{R}^{2N})}.
    \end{split}
\end{equation*}
The first term $\Vert{ \phi_{\varepsilon, k} - \psi^{1\mathrm{D}}_k}\Vert_{L^2_U(\mathbb{R}^{2N})} $ vanishes as $\varepsilon$ goes to $0$ because of the strong convergence of Lemma \ref{normal}. 
It is clear that $\psi^{1\mathrm{D}}_k( e^{-\mathrm{i}\alpha\sqrt{\varepsilon} S_{}} -1 )$ is dominated by $2\psi^{1\mathrm{D}}_k$, so using the dominated convergence theorem, the second term $\Vert{\psi^{1\mathrm{D}}_k( e^{-\mathrm{i}\alpha\sqrt{\varepsilon} S_{}} - 1 )}\Vert_{L^2_U(\mathbb{R}^{2N})}  $ also vanishes as $\varepsilon$ goes to $0$. 
Hence, $$\Psi^{\mathrm{2D}}_{k} - \psi^{1\mathrm{D}}_k U_{\varepsilon} \to 0 $$ 
strongly in $L^2(\mathbb{R}^{2N})$.

Recall the definition of $\varphi^{\varepsilon}_{k} $ in~\eqref{eq:proj1D}. We can rewrite it with the same notation as in the proof of Proposition \ref{ConvPhi}:
\begin{equation*}
    \varphi_{k}^{\varepsilon}(x_1,\dots,x_N) =     \int_{\mathbb{R}^N}{\phi_{\varepsilon,k}^{}} (\mathbf{x}_1,\dots,\mathbf{x}_N) U_{1}^2({y}_1,\dots,{y}_N)\differential y_1\cdots\differential y_N.
\end{equation*}
Meanwhile, we can rewrite the eigenfunction $ \psi^{1\mathrm{D}}_{k}$ in the similar form:
\begin{equation*}
    \psi^{1\mathrm{D}}_{k}(x_1,\dots,x_N)  = \int_{\mathbb{R}^N} \psi^{1\mathrm{D}}_{k}(x_1,\dots,x_N) U^2_{1}(y_1,\dots,y_N) \d y_1\cdots \d y_N .
\end{equation*}
Using the Cauchy-Schwarz inequality, we have
\begin{equation*}
    \begin{split}
        \int_{\mathbb{R}^N\backslash \mathbb{D}^{1\mathrm{D}}} \left| \varphi^{\varepsilon}_{k} - \psi^{1\mathrm{D}}_k \right|^2  & \le  \int_{\mathbb{R}^N\backslash \mathbb{D}^{1\mathrm{D}}} \LR{\int_{\R^N} \left|\phi_{\varepsilon,k} - \psi^{1\mathrm{D}}_k \right| U_{1}^2  \d y_1\cdots \d y_N }^2 \\
        & \le \int_{\mathbb{R}^N\backslash \mathbb{D}^{1\mathrm{D}}} \LR{\int_{\R^N}  \left|\phi_{\varepsilon,k} - \psi^{1\mathrm{D}}_k \right|^2 U_1^2\d y_1\cdots \d y_N }.
    \end{split}
\end{equation*}
Similarly to their derivatives, we then obtain
\begin{equation}\label{controlprojection}
    \norm{\varphi^{\varepsilon}_{k} - \psi^{1\mathrm{D}}_k}_{H^1(\mathbb{R}^N\backslash \mathbb{D}^{1\mathrm{D}})} \le \norm{\phi_{\varepsilon,k} - \psi_k^{1\mathrm{D}}}_{H^1_U(\Lambda_0)}.
\end{equation}
Since $ \{\phi_{\varepsilon,k}\}_{\varepsilon} $ converges strongly to $\psi^{1\mathrm{D}}_k$ in $H^{1}_{U}(\Lambda_0)$ by Lemma \ref{cptH1} above, the control~\eqref{controlprojection} implies $$\varphi_k^{\varepsilon} \to \psi_k^{1\mathrm{D}}$$ strongly in $H^1(\mathbb{R}^N\backslash \mathbb{D}^{1\mathrm{D}})$, which concludes the proof of Theorem \ref{thmEF}.
\end{proof}

\newpage
%
%

\end{document}